\documentclass[11pt,reqno]{amsart}

\usepackage{color}
\usepackage{stmaryrd}
\usepackage{cases}
\usepackage{amssymb}
\usepackage{amsmath}
\usepackage{amsfonts}
\usepackage{graphicx}
\usepackage{amsmath,amstext,amsbsy,amssymb, color}

\usepackage[colorlinks=true, linkcolor=blue, citecolor=magenta, urlcolor=blue, backref=page]{hyperref}

\hypersetup{colorlinks,citecolor=blue,filecolor=black,linkcolor=blue,urlcolor=blue}
\usepackage{latexsym}
\usepackage{amsmath}
\usepackage{amssymb}
\usepackage{bm}
\usepackage{graphicx}
\usepackage{wrapfig}
\usepackage{fancybox}

\usepackage{geometry}
\newtheorem{theorem}{Theorem}[section]
\newtheorem{lemma}[theorem]{Lemma}
\newtheorem{proposition}[theorem]{Proposition}
\newtheorem{corollary}[theorem]{Corollary}
\theoremstyle{definition}

\theoremstyle{remark}
\newtheorem{remark}[theorem]{Remark}

\numberwithin{equation}{section} \errorcontextlines=0

\newcommand{\ZZ}{\mathbb Z}

\newcommand{\Imm}{\mathrm{Imm}}

\newcommand{\ot}{\otimes}

\newcommand{\si}{\sigma}

\newcommand{\gl}{\mathfrak{gl}}

\newcommand{\Mat}{\mathrm{Mat}}
\newcommand{\End}{\mathrm{End}}
\newcommand{\Ber}{\mathrm{Ber}}

\newcommand{\Hc}{\mathcal{H}}
\newcommand{\Ec}{\mathcal{E}}
\newcommand{\Tc}{\mathcal{T}}
\newcommand{\ve}{\varepsilon}
\newcommand{\bal}{\begin{aligned}}
\newcommand{\eal}{\end{aligned}}
\newcommand{\beq}{\begin{equation}}
\newcommand{\eeq}{\end{equation}}
\newcommand{\ben}{\begin{equation*}}
\newcommand{\een}{\end{equation*}}
\newcommand{\Sym}{\mathfrak S}
\newcommand{\CC}{\mathbb{C}}

\begin{document}
	
	\title[Quantum super Littlewood correspondences]
	{Quantum super  Littlewood correspondences}
	
	\author{Naihuan Jing}
	\address{Department of Mathematics, North Carolina State University, Raleigh, NC 27695, USA}
	\email{jing@ncsu.edu}
	
		\author{Yinlong Liu}
	\address{Department of Mathematics, Shanghai University, Shanghai 200444, China}
	\email{yinlongliu@shu.edu.cn}
	
	\author{Jian Zhang}
	\address{School of Mathematics and Statistics,
		Central China Normal University, Wuhan, Hubei 430079, China}
	\email{jzhang@ccnu.edu.cn}
	
	\thanks{{\scriptsize
			\hskip -0.6 true cm MSC (2020): Primary: 17A70 Secondary: 17B35, 15A15, 17B66, 05E10
    }}

\begin{abstract}
In this paper, we study the Littlewood theory associated with the quantum super immanants and supersymmetric polynomials, including both the super case and the quantum generalization. In the setting of quantum super Schur-Weyl duality between the quantum superalgebra $U_q(\gl_{m|n})$ and the Iwahori-Hecke algebra $\Hc_r$ of type A, we explicitly construct basis vectors of the ($U_q(\gl_{m|n})$, $\Hc_r$)-bimodule on the tensor product space $(\CC^{m|n})^{\ot r}$.
Using this construction, we interpret the quantum super immanants via weight spaces of covariant tensor representations of $U_q(\gl_{m|n})$.
\end{abstract}
	\maketitle
\section{Introduction}
In the study of the representation theory of the symmetric group and the general linear group,
I.Schur introduced the notion of immanants in his fundamental work \cite{S}, which  plays an important role in understanding the representation theory of these groups. Schur applied a prominent method to obtain the finite dimensional irreducible representations of the general linear group, or equivalently of its Lie algebra from those of symmetric group due to F.Frobenius and A.Young. This result is now known as {\it Schur-Weyl duality}. And then D.E.Littlewood \cite{Li,LR} formulated three correspondences between Schur polynomials and immanants, which explicitly reveal the representation-theoretic information from immanants. The Littlewood correspondences provide an alternative  perspective, different from Schur-Weyl duality, for understanding representations of the symmetric group and the general linear group. Continuing \cite{JLZ1}, we established the quantum analog of Littlewood correspondences for the quantum coordinate algebra. After \cite{JLZ1}, we intend to extend the Littlewood correspondences to the superalgebra.

In this paper, we continue these preceding approaches to establish complete quantum super Littlewood correspondences. Using the primitive idempotents of the Hecke algebra, we introduce the notion of quantum super immanants in the quantum coordinate superalgebra. From the representation theory of the Hecke algebra and the property of supersymmetric polynomials, we prove that the first two Littlewood correspondences hold in this setting.

Littlewood's third correspondence says that the Schur polynomials $S_{\lambda}$ can be explicitly expressed in terms of normalized immanants:
\begin{equation}
S_{\lambda}(\omega_1, \ldots, \omega_n)=\sum_{I} \frac{\Imm_{\chi^{\lambda}}(X_I)}{\alpha(I)},
\end{equation}
where $\omega_i$ are the characteristic roots of the matrix $X$. Unfortunately, these roots are no longer useful for quantum super case, as the concepts of characteristic roots and eigenvalues do not apply to quantum supermatrices. To obtain these  pairwise commutative elements within quantum coordinate superalgebra, we construct a system of equations from the characteristic function (Berezinian).
Solving these equations over an algebraic closure field yields the unique elements,  thereby establishing the quantum super analog of Littlewood correspondence III. Besides, we compute the quantum super Cayley-Hamilton theorem for the case $m=n=1$. The super Cayley-Hamilton theorem ($q\rightarrow 1$) was proved in \cite{UM,UM2}. We conjecture that the general statement also holds in the quantum super setting.

The immanants also have another connection with representation theory. Kostant \cite{Ko} interpreted the immanants using 0-weight spaces for all representations of the special linear group $\mathrm{SL}(n,\CC)$ and provided a trace formula for the $\lambda$-immanants. In this paper, we aim to obtain a supertrace formula for quantum super immanants and interpret it via weight spaces of covariant tensor representations of the quantum superalgebra $U_q(\gl_{m|n})$.

Using the quantum super Schur-Weyl duality developed by Moon \cite{Moo} and Mitsuhashi \cite{Mit}, we consider the actions of the quantum superalgebra $U_q(\gl_{m|n})$ and the Iwahori-Hecke algebra $\Hc_r$ on the tensor product space $(\CC^{m|n})^{\ot r}$. The actions of the generators of $U_q(\gl_{m|n})$ on the Gelfand-Tsetlin bases have been given for essentially typical representations in \cite{PSV} and for covariant tensor representations in \cite{Mo2,SJ,FSZ,Lu}.
Combining these Gelfand-Tsetlin bases with the Young orthonormal basis of Hecke algebra $\Hc_r$,  we explicitly construct  basis vectors  of the $(U_q(\gl_{m|n}),\Hc_r)$-bimodule on $(\CC^{m|n})^{\ot r}$.
Consequently, the weight space interpretation of quantum super immanants follows from the approach in  \cite{JLZ1}.

The paper is organized as follows. In Section 2, we prepare the basic facts of the quantum superalgebra and its coordinate superalgebra. In Section 3, we recall the property of the primitive idempotents of Hecke algebra of type A. We introduce the notion of quantum super immanants and give a supertrace formula in Section 4. In Section 5, we establish the Littlewood correspondences I-III for quantum super immanants.

\section{The quantum superalgebra and its coordinate superalgebra}
Let $\CC^{m|n}$ be the complex super vector space spanned by its $\ZZ_2$-graded basis $\{e_1,\ldots,e_{m+n}\}$. For any $1\leq i\leq m+n$, the parity of $e_i$ is defined as
\[
\bar{i}=\left\{\begin{array}{cc}
0&\text{ if }i\leq m,\\
1&\text{ if }i> m.
\end{array}\right.
\]
The endomorphisms of $\CC^{m|n}$ naturally have $\ZZ_2$-graded basis $\{e_{ij}\mid 1 \leq i,j \leq m+n\}$, where the parity of $e_{ij}$ is $\bar i+\bar j$.

Let $q_i=q^{1-2\bar i}$ for $̄i \in [m + n] = \{1, \ldots,  m + n\}$.
Let $R$ be the matrix in $\mathrm{End}(\CC^{m|n}) \ot\mathrm{End}(\CC^{m|n})$:
\begin{equation}\label{R}
  R=\sum_{i=1}^{m+n}q_i e_{ii}\ot e_{ii}+\sum_{ i\neq j} e_{ii}\ot e_{jj}+(q-q^{-1})\sum_{i<j} (-1)^{\bar j } e_{ij}\ot e_{ji}.
\end{equation}
The  $R$ satisfies the well-known Yang-Baxter equation (YBE):
\begin{equation}\label{YBeq}
  R_{12} R_{13} R_{23}=R_{23} R_{13} R_{12},
\end{equation}
where  $R_{i j} \in \mathrm{End}(\CC^{m|n})  \otimes \mathrm{End}(\CC^{m|n}) \otimes \mathrm{End}(\CC^{m|n}) $  acts on the  $i$th and  $j$th copies of  $\CC^{m|n}$  as  $R$  does on  $\CC^{m|n} \otimes \CC^{m|n}$. In the superalgebra $\mathrm{End}(\CC^{m|n})^{\ot 2}$, we define the permutation operator
\beq
P=\sum_{i,j=1}^{m+n}(-1)^{\bar j}e_{ij}\ot e_{ji}.
\eeq

 There are two $R$-matrices $R^{\pm}$ associated with $R$ by $R^+=PRP$, $R^-=R^{-1}$. Explicitly
\begin{align}
  R^+&=\sum_{i=1}^{m+n}q_i e_{ii}\ot e_{ii}+\sum_{ i\neq j} e_{ii}\ot e_{jj}+(q-q^{-1})\sum_{i>j} (-1)^{\bar j } e_{ij}\ot e_{ji},\\
  R^-&=\sum_{i=1}^{m+n}q_i^{-1} e_{ii}\ot e_{ii}+\sum_{ i\neq j} e_{ii}\ot e_{jj}-(q-q^{-1})\sum_{i<j} (-1)^{\bar j } e_{ij}\ot e_{ji}.
\end{align}

 Let $\mathbf P$ be the free $\ZZ$-lattice of integral weight  with the canonical basis $\{\epsilon_1,\ldots,\epsilon_{m+n}\}$, i.e. $\mathbf P=\oplus_{i=1}^{m+n}\ZZ \epsilon_i$, endowed with a symmetric bilinear form $\langle \epsilon_i,\epsilon_j \rangle=\delta_{ij}$. Let $\Pi=\{\alpha_i=\epsilon_i-\epsilon_{i+1}\mid i=1,\ldots,m+n-1 \}$ and $\Phi=\{\epsilon_i-\epsilon_j \mid 1 \leq i\neq j \leq m+n \}$. Then $\Phi$ realizes the root system of  type $A(m-1,n-1)$ with $\Pi$ a basis of simple roots. Let $H_i=\alpha_i$ if $i \neq m$ and $H_m=\epsilon_m+\epsilon_{m+1}$.

 The quantum superalgebra $U_q(\mathfrak{gl}_{m|n})$ is generated by  $E_i,F_i \ (1\leq i\leq m+n-1)$ and $q^{h} \ (h\in \mathbf P)$ with parity $p(q^h)=p(E_i)=p(F_i)=0$ $(h\in \mathbf P, i\neq m)$, $p(E_m)=p(F_m)=1$ and satisfies the following relations:
\beq
\bal
&q^0=1, \qquad q^{h_1+h_2}=q^{h_1}q^{h_2}, \ \forall h_1,h_2\in \mathbf P,\\
&q^{h}E_i q^{-h}=q^{\langle h,\alpha_i \rangle}E_i,
\qquad
q^{h}F_{i}q^{-h}=q^{-\langle h,\alpha_i \rangle}F_{i},\\
&E_i F_j - (-1)^{p(E_i)p(F_j)} F_j E_i = \delta_{i,j} \frac{q^{(-1)^{\bar i} H_i} - q^{-(-1)^{\bar i} H_i}}{q_i - q_i^{-1}},\\
&E_i E_j - (-1)^{p(E_i)p(E_j)} E_j E_i = 0 \ \text{if} \ |i - j| \geq 2,\\
&F_i F_j - (-1)^{p(F_i)p(F_j)} F_j F_i = 0 \ \text{if} \ |i - j| \geq 2,\\
&E_i^2 E_j - (q_i + q_i^{-1}) E_i E_j E_i + E_j E_i^2 = 0 \ \text{if} \ |i - j| =1, i \neq m,\\
&F_i^2 F_j - (q_i + q_i^{-1}) F_iF_j F_i + F_j F_i^2 = 0 \ \text{if} \ |i - j| =1, i \neq m,\\
&E_m^2=F_m^2=0,\\
&E_m E_{m-1} E_m E_{m+1} + E_m E_{m+1} E_m E_{m-1} + E_{m-1} E_m E_{m+1} E_m \\
&+ E_{m+1} E_m E_{m-1} E_m - (q + q^{-1}) E_m E_{m-1} E_{m+1} E_m = 0,\\
&F_m F_{m-1} F_m F_{m+1} + F_m F_{m+1} F_m F_{m-1} + F_{m-1} F_m F_{m+1} F_m \\
&+ F_{m+1} F_m F_{m-1} F_m - (q + q^{-1}) F_m F_{m-1} F_{m+1} F_m = 0.
\eal
\eeq

The quantum superalgebra $U_q(\mathfrak{gl}_{m|n})$ also has the following RTT presentation. It is generated by $t_{ij}^{\pm}, i,j=1,\ldots,m+n$ with parity $\bar i+ \bar j$ and the following relations:
\beq
\bal
t_{ij}^{+}&=t_{ji}^{-}=0,\ 1 \leq j<i \leq m+n\\
t_{ii}^{+}t_{ii}^{-}&=t_{ii}^{-}t_{ii}^{+}=1,\ 1\leq i \leq m+n\\
RT_1^{\pm}T_2^{\pm}&=T_2^{\pm}T_1^{\pm}R\\
RT_1^{+}T_2^{-}&=T_2^{-}T_1^{+}R
\eal
\eeq
where the matrices $T^{\pm}=(t_{ij}^{\pm})$.
The isomorphism between these two presentations is given by
\beq
\bal
t_{ij}^+=\left\{\begin{array}{ccc}
q^{\epsilon_i}_i, \ &i=j\\
(-1)^{\bar i}(q-q^{-1})q^{\epsilon_i}_i\hat{E}_{ji}, \ &i<j\\
0, \ &i>j\\
\end{array}\right.
\eal
\eeq
and
\beq
\bal
t_{ij}^-=\left\{\begin{array}{ccc}
q^{-\epsilon_{i}}_i, \ &i=j\\
-(-1)^{\bar i}(q-q^{-1})\hat{E}_{ji}q^{-\epsilon_j}_j, \ &i>j\\
0, \ &i<j.\\
\end{array}\right.
\eal
\eeq
where   $\hat{E}_{ij} (i \neq j)$ are determined by the following equations
\beq
\bal
&\hat{E}_{i,i+1}=E_i,\ \ \ \hat{E}_{i+1,i}=F_i,\\
&\hat{E}_{ij}=(-1)^{\bar k}(\hat{E}_{ik}\hat{E}_{kj}-q_k^{\pm 1}\hat{E}_{kj}\hat{E}_{ik}), \ \ i \lessgtr k \lessgtr j,\\
\eal
\eeq

A Hopf superalgebra structure of $U_q(\mathfrak{gl}_{m|n})$ is given by the comultiplication $\Delta: U_q(\mathfrak{gl}_{m|n})\rightarrow U_q(\mathfrak{gl}_{m|n}) \ot U_q(\mathfrak{gl}_{m|n})$ such that
 \beq
\bal
\Delta(E_i)&=E_i\ot 1+k_i \ot E_i,\\
\Delta(F_i)&=F_i\ot k_i^{-1}+ 1 \ot F_i,\\
\Delta(q^h)&=q^h\ot q^h,\\
\Delta(T^{\pm})&=T^{\pm}\ot T^{\pm},
\eal
\eeq
where we set $k_i=q_i^{H_i}$. The antipode
$S: U_q(\mathfrak{gl}_{m|n})\rightarrow U_q(\mathfrak{gl}_{m|n})$ is given by
\beq
S(E_i)=-k_i^{-1}E_i, \ \ S(F_i)=-F_ik_i, \ \ S(q^h)=q^{-h},\ \ S(T^{\pm})=(T^{\pm})^{-1},
\eeq
and the counit $\varepsilon: U_q(\mathfrak{gl}_{m|n})\rightarrow \CC(q)$ by
\beq
\varepsilon(E_i)=\varepsilon(F_i)=0, \ \ \varepsilon(q^h)=1,\ \ \varepsilon(T^{\pm})=I. \\
\eeq

There is a fundamental representation of $U_q(\mathfrak{gl}_{m|n})$ on $\CC^{m|n}$:
\beq
\bal
\rho: U_q(\mathfrak{gl}_{m|n})&\rightarrow \End(\CC^{m|n})\\
T^{\pm}&\mapsto R^{\pm},
\eal
\eeq
and $\rho(E_i)=e_{i,i+1},\ \rho(F_i)=e_{i+1,i},\ \rho(q^h)=\sum_{i=1}^{m+n}q^{\langle h,\epsilon_i  \rangle}e_{ii}$. We can define the action $\rho_r \ (r >1)$ of $U_q(\mathfrak{gl}_{m|n})$  on tensor space $(\CC^{m|n})^{\ot r} $ by its Hopf superalgebra structure.

The {\it  quantum coordinate superalgebra} $A_{q}(\Mat_{m|n})$ is the finite dual of quantum superalgebra $U_q(\mathfrak{gl}_{m|n})$,  generated by  $x_{i j}, 1 \leq i, j \leq m+n$  subject to the quadratic relations defined by the matrix equation
\begin{equation}\label{RT}
  R X_{1} X_{2}=X_{2} X_{1} R
\end{equation}
in  $A_{q}(\Mat_{m|n}) \otimes \mathrm{End}(\CC^{m|n})^{\ot 2}$ and the parity of $x_{ij}$ is $\bar i+\bar j$.
In terms of entries, the relations can be written as
\begin{align}\label{e:reln}
&x_{ik}^2=0,\ \bar i+\bar k=1,  \\
&x_{i l} x_{i k}=(-1)^{(\bar i+\bar l)(\bar i+\bar k)}q_ix_{i k} x_{i l}, \\
&x_{j k} x_{i k}=(-1)^{(\bar i+\bar k)(\bar j+\bar k)}q_k x_{i k} x_{j k},\\
&x_{j k} x_{i l}=(-1)^{(\bar j+\bar k)(\bar i+\bar l)}x_{i l} x_{j k},\\
&x_{j l} x_{i k}=(-1)^{(\bar j+\bar l)(\bar i+\bar k)}x_{i k} x_{j l}+(q-q^{-1})(-1)^{\bar j(\bar i+\bar k)+\bar i \bar k} x_{i l} x_{j k},
\end{align}
where  $i<j$  and  $k<l$.
The superalgebra  $A_q(\Mat_{m|n})$ is a super bialgebra under the comultiplication
 $A_q(\Mat_{m|n})\longrightarrow A_q(\Mat_{m|n})\ot A_q(\Mat_{m|n})$
 defined by
\begin{equation}
\Delta(X)=X\otimes X,
\end{equation}
and the counit given by $\varepsilon(X)=I$.

There exists a dual pairing $(\ ,\ )$ of super bialgebra $A_{q}(\Mat_{m|n})$ and $U_q(\mathfrak{gl}_{m|n})$:
\beq
(X_1,T_2^{\pm})=R^{\pm}.
\eeq
This duality equips $(\CC^{m|n})^{\ot r}$  a $A_{q}(\Mat_{m|n})$-comodule structure  by
\beq
\bal
 \Pi_r: (\CC^{m|n})^{\ot r} &\rightarrow  (\CC^{m|n})^{\ot r}\ot A_{q}(\Mat_{m|n})\\
   e_{j_1}\ot \cdots \ot e_{j_r}&\mapsto \sum_{(i_1,\ldots,i_r)}(-1)^{\sum\limits_{a<b} \bar  i_b(\bar i_a+\bar j_a)} e_{i_1}\ot \cdots \ot e_{i_r}\ot x_{i_1,j_1}\cdots x_{i_r,j_r},\\
\eal
\eeq
where the sum is over all sequence $(i_1\ldots,i_r)$ of $[m+n]$.

\section{Hecke algebra $\mathcal H_r$ and the Yang-Baxter equation}
The  Hecke algebra  $\Hc_r$ is the associative algebra generated by elements
$T_1,\ldots,T_{r-1}$ subject to the relations
\begin{align}\label{e:Hecke}
&(T_i-q)(T_i+q^{-1})=0,\\
&T_{i}T_{i+1}T_{i}=T_{i+1}T_{i}T_{i+1},\\
&T_iT_j=T_jT_i \quad\text{for\ \  $|i-j|>1$}.
\end{align}

For each $1\leq i\leq r-1$, let $\si_{i}=(i, i+1)$ be the adjacent transposition in the symmetric group $\Sym_{r}$.
Choose a reduced decomposition
$\si=\si_{i_{1}} \dots \si_{i_{l}}$ of any element
$\si \in \Sym_{r}$ and  $T_{\si}=T_{i_{1}} \dots T_{i_{l}}$ is a reduced expression that is independent of the choice of reduced expressions of $\sigma$.

Define $\check{R}=PR$. Explicitly
\begin{equation}\label{PR}
  \check{R}=\sum_{i=1}^{m+n}(-1)^{\bar i} q_i e_{ii}\ot e_{ii}+\sum_{ i\neq j} (-1)^{\bar j} e_{ij}\ot e_{ji}+(q-q^{-1})\sum_{ i>j} e_{ii}\ot e_{jj}.
\end{equation}
The matrix $\check{R}$ satisfies the following relations:
\ben
\bal
&\check{R}_{k}\check{R}_{k+1}\check{R}_{k}=\check{R}_{k+1}\check{R}_{k}\check{R}_{k+1},\\
&(\check{R}_{k}-q)(\check{R}_{k}+q^{-1})=0,
\eal
\een
where $\check{R}_{k}=P_{k,k+1}R_{k,k+1}, 1 \leq k \leq r-1$.
Then there is a natural representation of the Hecke algebra $\Hc_{r}$ on the
tensor space  $(\CC^{m|n})^{\ot r}$ defined by
\beq
T_{k}\mapsto \check{R}_{k},\qquad k=1,\dots,r-1.
\eeq
And the RTT relation \eqref{RT} of $A_{q}(\Mat_{m|n})$ can be rewritten equivalently as
\beq
\check{R}X_1X_2=X_1X_2 \check{R}.\\
\eeq

For generic $q$, the irreducible representations of $\mathcal{H}_r$ over $\mathbb{C}(q)$ are parameterized by partitions of $r$. Given a partition $\lambda \vdash r$, we denote $(\xi^{\lambda},V^{\lambda})$ the corresponding irreducible representation.
The module $V^{\lambda}$ can be realized by standard Young tableaux  of shape $\lambda$.

Let $\{v_{\mathcal T}\}$ be the normalized orthogonal Young basis of $V^{\lambda}$ indexed by standard Young tableaux $\Tc$ of shape $\lambda$, endowed the inner product $\langle \ , \ \rangle$ satisfying the property
that $\langle T_{\sigma}u, v\rangle=\langle u, T_{\sigma^{-1}}v\rangle$.
Let $c_k(\Tc)=j-i$ denote the content of the box $(i,j)$ of $\lambda$ occupied by $k$ in $\Tc$.
Denote by $\mathrm{SYT}(\lambda)$ the set of all standard Young tableaux of shape $\lambda$.

For any integer $n\geq 0$, set
\[
[n]_q=\frac{q^n-q^{-n}}{q-q^{-1}},\qquad  [n]_q!=[1]_q\cdots[n]_q.
\]
Then the action of $\mathcal H_r$ is given by
\begin{align}\label{Young-orthogonal-form}
T_iv_{\Tc}=\frac{q^{d_i(\Tc)}}{[d_i(\Tc)]_q}v_{\Tc}+
\frac{\sqrt{[d_i(\Tc)+1]_q[d_i(\Tc)-1]_q}}{[d_i(\Tc)]_q}v_{s_i\Tc}
\end{align}
where $d_i(\Tc)=c_{i+1}(\Tc)-c_i(\Tc)$.

The {\it character} of the representation $V^{\lambda}$ is defined as
\begin{equation}\label{ch formula}
\chi_q^{\lambda}
=\sum_{\si \in \mathfrak{S}_r}T_{\si}\mathcal{E}_{\Tc}^{\lambda}T_{\si^{-1}},
\end{equation}
where $\mathcal{E}_{\Tc}^{\lambda}$ is the primitive idempotent of $\mathcal{H}_r$ associated to standard Young tableau $\Tc$.

The primitive idempotent $\mathcal{E}_{\Tc}^{\lambda}$  of $\mathcal{H}_r$ can be computed by
\begin{equation}\label{e:idem}
\mathcal{E}_{\Tc}^{\lambda}=\frac{1}{c_{\lambda}}\sum_{\si \in \Sym_r}\langle T_{\si^{-1}}v_{\Tc},v_{\Tc}\rangle T_{\si},
\end{equation}
where $c_{\lambda}$ is   the Schur element and given  by Steinberg's formula
 \beq
 c_{\lambda}=\prod_{\alpha\in \lambda}q^{c(\alpha)}[h(\alpha)]_q.
 \eeq
 We will briefly write $\mathcal{E}_{\Tc}=\mathcal{E}_{\Tc}^{\lambda}$ if the shape of $\Tc$ is clear from the context.
We have the following useful lemma from the Lemma 1.1.1 in \cite{Mo}.
\begin{lemma}\label{idem-lemma}
Let $\Tc$ be a standard tableau of shape $\lambda$ and $1 \leq a \leq r-1$, then
\beq
\bal
\mathcal{E}_{\Tc}^{\lambda}(T_a-\frac{q^{d_a(\Tc)}}{[d_a(\Tc)]_q})&=\mathcal{E}_{\Tc}^{\lambda}T_a \mathcal{E}_{(a,a+1)\Tc}^{\lambda},
\eal
\eeq
where we suppose that $\mathcal{E}_{(a,a+1)\Tc}^{\lambda}=0$ if the tableau $(a,a+1)\Tc$ is not standard.
\end{lemma}

The Jucys-Murphy elements $y_1,\ldots,y_r$ of $\Hc_r$ are defined by $y_1=1$ and
\begin{align}
  y_k=1+(q-q^{-1})(T_{(1,k)}+T_{(2,k)}+\cdots+T_{(k-1,k)}), \quad k=2,\ldots,r.
\end{align}
These elements generate the maximal commutative subalgebra of $\Hc_r$ and
\begin{align}\label{JM_rela}
  y_k\mathcal{E}_{\Tc}^{\lambda}=\mathcal{E}_{\Tc}^{\lambda}y_k=q^{2c_k(\Tc)}\mathcal{E}_{\Tc}^{\lambda}.
\end{align}

Next we recall the fusion procedure of primitive idempotents in \cite{DJ2,Ch, IMO,N}.   Denote by $\Tc^-$ the standard tableau obtained from $\Tc$ by removing the box $\alpha$ occupied by $r$. The shape of $\Tc^-$ is denoted by $\mu$. We have the recurrence relation:
\ben
\mathcal{E}_{\Tc}^{\lambda}=\mathcal{E}_{\Tc^-}^{\mu}\frac{(y_r-q^{2a_1})\cdots (y_r-q^{2a_l})}{(q^{2c_{\alpha}}-q^{2a_1})\cdots (q^{2c_{\alpha}}-q^{2a_l})},
\een
where $a_1,\ldots,a_l$ are the contents of all addable boxes of $\mu$ except for $\alpha$, while $c_{\alpha}$ is the content of the latter.

On the other hand, take $r$ complex variables $z_1,\ldots,z_r$ and consider the $\Hc_r$-valued rational function defined by
\ben
\phi_{\Tc}(z_1,\ldots,z_r)=\prod^{\rightarrow}_{(i,j)} T_{j-i}(q^{2c_i(\Tc)}z_i,q^{2c_j(\Tc)}z_j),
\een
where the product is taken in the lexicographical order on the set of pairs $(i,j)$ with $1 \leq i <j \leq r$. And for any $k=1,\ldots,r-1$,  the rational functions $T_k(x,y)$ in two variables $x,y$ be defined by
\ben
T_k(x,y)=T_k+\frac{q-q^{-1}}{x^{-1}y-1}.
\een
Then the primitive idempotent $\Ec_{\Tc}^{\lambda}$ can be obtained by the consecutive evaluations
\beq\label{fusion-procedure}
\Ec_{\Tc}^{\lambda}=\frac{1}{c_{\lambda^T}}\phi_{\Tc}(z_1,\ldots,z_r)T_0^{-1}|_{z_1=1}\cdots |_{z_r=1},
\eeq
where we denote $T_0$ the unique longest element in $\Hc_r$ and $T_0$ satisfies the relations:
\beq\label{T0-relation}
T_0T_j=T_{r-j}T_0, \ \ 1 \leq j\leq r-1.
\eeq

For some $\mu\vdash r$, let $\mathfrak S_{\mu}$ be a Young subgroup of the symmetric group $\mathfrak{S}_r$ of type $\mu$. Denote by $K$ the set of minimal representatives of cosets in $\mathfrak{S}_r/\mathfrak S_{\mu}$. Let $\mathcal{H}_{\mu}$ be the corresponding  parabolic subalgebra of $\mathcal{H}_r$, and $V$ is a representation of $\mathcal{H}_{\mu}$.
The induced representation of $V$ is defined by
$$\mathrm{Ind}_{\mathcal{H}_{\mu}}^{\mathcal{H}_r}(V)=\mathcal{H}_r\ot_{\mathcal{H}_{\mu}} V.$$
We briefly write $\mathrm{Ind}(V)=\mathrm{Ind}_{\mathcal{H}_{\mu}}^{\mathcal{H}_r}(V)$.

\section{Quantum super immanants and Schur-Weyl duality}
In this section, we will introduce the notion of quantum super immanants in the quantum coordinate superalgebra $A_q(\Mat_{m|n})$ with any representation of the
Hecke algebra $\Hc_r$ and give a representation-theoretic interpretation for quantum super immanants.
\subsection{Quantum super immanants}
Let  $I=(i_1,\ldots, i_r)$ and $J=(j_1, \ldots ,j_r)$ be two multisets ($r$-tuples) of $[m+n]=\{1,\ldots,m+n\}$. Let $X=(x_{ij})$ be the generator matrix of $A_q(\Mat_{m|n})$.  We can write supermatrix $X$ in the block form:
\[
X = \begin{pmatrix}
X_{11} & X_{12}\\
X_{21} & X_{22}
\end{pmatrix},
\]
where $X_{11},X_{22}$ (resp. $X_{12},X_{21}$) are respectively, $m \times m$,  $n \times n$ (resp.  $m \times n$,  $n \times m$) dimensional matrices  with  even (resp. odd) entries.
The supertrace $str$ of  $X$ is
\[
str(X)=\sum_{i=1}^{m+n}(-1)^{\bar{i}}x_{ii}.
\]
For any  $a\in\{1,2,\ldots,r\}$ we will denote by $str_a$ the corresponding supertrace on the
superalgebra $\mathrm{End}((\mathbb{C}^{m|n})^{\ot r})$ which acts as $str$ on the $a$-th copy of $\mathrm{End}(\mathbb{C}^{m|n})$ and as the identity map on all
the other tensor factors.

Denote  $X^I_J$ the generalized submatrix of $X$ whose row (resp. column) indices belong to  $I$ (resp. $J$). If $I=J$, we briefly write $X_I$.
It will be convenient to use a standard notation for the matrix coefficients $A_{j_1,\dots,j_r}^{i_1,\dots,i_r}$ of an operator $A$ acting on the standard basis of  $\mathrm{End}((\mathbb{C}^{m|n})^{\ot r})$.
We denote
\begin{equation}
  A=\sum_{I,J}A_{j_1,\dots,j_r}^{i_1,\dots,i_r} e_{i_1j_1}\ot \cdots \ot e_{i_rj_r},
\end{equation}
where $e_{ij}$ be the standard basis of $\mathrm{End}(\mathbb{C}^{m|n})$.
Let $\{e_{i}^*|1\leq i\leq m+n\}$ be the basis of $\mathbb (\mathbb C^{m|n})^*$ dual to the basis $\{e_{i}|1\leq i\leq m+n\}$ of $\mathbb C^{m+n}$, then we can write the dual bases of $((\mathbb C^{m|n})^{\ot r})^{*}$ and $(\mathbb C^{m|n})^{\ot r}$ respectively:
\beq
\langle i_1,\dots ,i_r\mid=(e_{i_1}\otimes\cdots \otimes  e_{i_r})^*, \qquad \mid i_1,\dots,i_r\rangle =e_{i_1}\otimes\cdots \otimes e_{i_r},
\eeq
with parity $\bar I=\bar i_1+\ldots +\bar i_r$.

Then the  coefficients of the operator $A$ are given by
\beq
A_{j_1,\ldots,j_r}^{i_1,\ldots,i_r}=(-1)^{\gamma(I,J)}\langle i_1,\ldots ,i_r\mid A \mid j_1,\ldots,j_r\rangle,
\eeq
where
$\gamma(I,J)=\sum_{a} \bar i_a(\bar j_a+1)+\sum_{a<b}  \bar j_b(\bar i_a +\bar j_a)$. In particular,
\beq
\bal
A_{i_1,\ldots,i_r}^{i_1,\ldots,i_r}&=\langle i_1,\ldots ,i_r\mid A \mid i_1,\ldots,i_r\rangle,\\
str_{1,\ldots,r}(A)&=\sum_{I}(-1)^{\bar I}\langle i_1,\ldots ,i_r\mid A \mid i_1,\ldots,i_r\rangle,
\eal
\eeq

Let $V$ be any representation of $\Hc_r$ with the character $\chi^{V}$.
The {\it quantum super immanant} of $X^I_J$ associated to the representation $V$ is defined by
\begin{equation}\label{imm-def}
  \Imm_{\chi ^{V}}(X^I_J)=(-1)^{\sum_{k=1}^r\bar i_k\bar j_k}\langle i_{1},\ldots ,i_{r}\mid \chi^{V}X_1\cdots X_{r} \mid j_{1},\ldots,j_{r}\rangle.
\end{equation}
With equation \eqref{ch formula}, the quantum super immanants corresponding to representation $V^{\lambda}$ can be written as
\begin{equation}
  \Imm_{\chi_{q}^{\lambda}}(X^I_J)=(-1)^{\sum_{k=1}^r\bar i_k\bar j_k}\sum_{\si\in \Sym_r}\langle i_{1},\ldots ,i_{r} \mid \check{R}_{\si} \mathcal{E}_{\Tc}^{\lambda}\check{R}_{\si^{-1}} X_1\cdots X_{r} \mid j_{1},\ldots,j_{r}\rangle.
\end{equation}

It follows from equation \eqref{PR}  that $\check{R}(v_i\ot v_j)=(-1)^{\bar i\bar j}v_j\ot v_i$ for $i<j$. Therefore, for any $\si \in \mathfrak{S}_r$ and subset $I=(1\leq i_1 < \ldots < i_r\leq m+n)$, we have
\begin{equation}
\begin{aligned}
  \check{R}_{\si^{-1}} \mid i_1,\dots,i_r\rangle=(-1)^{\bar I_{\si^{-1}}}\mid i_{\si_1},\dots i_{\si_r}\rangle, \\
\langle i_1,\dots,i_r \mid \check{R}_{\si}=(-1)^{\bar I_{\si^{-1}}}\langle i_{\si_1},\dots i_{\si_r}\mid,
\end{aligned}
\end{equation}
where
\ben
\bar I_{\si}=\sum_{k<t, \atop \si(k)> \si(t)}\bar i_k \bar i_t.
\een

It is convenient to write $I=(1^{\alpha_1},\ldots,(m+n)^{\alpha_{m+n}})$ to specify the multiplicity $\alpha_i$ of $i$ in the non-decreasing multiset $I=(i_1\leq \ldots \leq i_r)$. Here $\alpha_i=\alpha_i(I)= \mathrm{Card}\{j\in I \mid j=i\}$. Let  $$\alpha({I})=\alpha_1!\alpha_2 ! \cdots \alpha_{m+n} !$$
For any integer $z\geq 0$, set $q'_i=(-1)^{\bar i}q_i$. Denote $\mathrm{q}=(q_1,\ldots,q_{m+n})$ and $\mathrm q'=(q'_1,\ldots,q'_{m+n})$, then $(\mathrm q')^2=\mathrm q^2=(q_1^2,\ldots,q_{m+n}^2)$.
The $q$-numbers and $q$-factorials are defined as
\[ (z)_{q_i}=\frac{(q_i)^z-1}{q_i-1}, \qquad (z)_{q_i}!=(1)_{q_i}\cdots(z)_{q_i}.
\]
Let $$\alpha_{\mathrm q}({I})=(\alpha_1)_{ q_1}!(\alpha_2)_{q_2} ! \cdots (\alpha_{m+n})_{q_{m+n}} !$$
Denote by $$\mathfrak S_{I}=\mathfrak S_{\alpha_1}\times\mathfrak S_{\alpha_2} \times \cdots \times \mathfrak S_{\alpha_{m+n}}$$  the Young subgroup  of $\mathfrak{S}_r$, and $\mathfrak S_{\alpha_j}=1$ if $\alpha_j=0$.
Denote by $H(m,n)$ the set of partitions $\lambda$ such that $\lambda_{m+1}\leq n$ and let
\ben
H(m,n;r)=\{\lambda\in H(m,n)\mid \lambda \vdash r\}.
\een
\begin{proposition}\label{imm-char}
  Let $I=(i_1\leq \ldots \leq i_r)$ be an ordered  multiset of $[m+n]$ and $X$ be the generator supermatrix of $A_q(\Mat_{m|n})$. Then for any $\lambda\vdash r$,
  \begin{equation}\label{imm-XI}
    \frac{\Imm_{\chi_q^{\lambda}}(X_I)}{\alpha_{\mathrm q^2}(I) }=\frac{(-1)^{\bar I}}{\alpha( {I})}\sum_{\si\in \Sym_r}\langle i_{\sigma_1},\dots ,i_{\sigma_r}\mid  \mathcal{E}_{\Tc}^{\lambda} X_1\cdots X_{r} \mid i_{\sigma_1},\dots ,i_{\sigma_r}\rangle,
  \end{equation}
 Moreover when $\lambda \notin H(m,n;r)$, $\Imm_{\chi_q^{\lambda}}(X_I)=0$.
\end{proposition}
\begin{proof}
We denote $\mathcal{E}_{\Tc}^{\lambda}$ by $\mathcal{E}_{\Tc}$.
By definition,
\begin{equation}
\bal
  \Imm_{\chi_q^{\lambda}}(X_I)&=(-1)^{\bar I}\sum_{\sigma\in \mathfrak{S}_r}\langle i_{1},\ldots ,i_{r}\mid \check{R}_{\si} \mathcal{E}_{\Tc}\check{R}_{\si^{-1}}X_1\cdots X_{r} \mid i_{1},\ldots,i_{r}\rangle\\
  &=(-1)^{\bar I}\sum_{\sigma\in \mathfrak{S}_r}\langle i_{1},\ldots ,i_{r}\mid \check{R}_{\si} \mathcal{E}_{\Tc}X_1\cdots X_{r} \check{R}_{\si^{-1}}\mid i_{1},\ldots,i_{r}\rangle.
\eal
\end{equation}

Let $\mathcal M(\Sym_r/\Sym_I)$ be the minimal length coset representative of $\Sym_r/\Sym_I$.
For any  $\tau \in\Sym_{\alpha_i}$,
\begin{align*}
\check{R}_{\tau}\mid i_1,\dots ,i_r\rangle=(q'_i)^{l(\tau)}\mid i_1,\dots ,i_r\rangle,\\
 \langle i_1,\dots ,i_r \mid \check{R}_{\tau^{-1}}= (q'_i)^{l(\tau)} \langle i_1,\dots ,i_r \mid.
\end{align*}
For any  $\mu \in \mathcal M(\Sym_r/\Sym_I)$,
\begin{align*}
  &\check{R}_{\mu^{-1}}  \mid i_1,\dots ,i_r\rangle= (-1)^{\bar I_{\mu^{-1}}}\mid i_{\mu_1},\dots ,i_{\mu_r}\rangle,\\
  & \langle i_1,\dots ,i_r\mid \check{R}_{\mu}=(-1)^{\bar I_{\mu^{-1}}} \mid i_{\mu_1},\dots ,i_{\mu_r}\rangle.
\end{align*}
Therefore,
\begin{equation*}
\bal
  \Imm_{\chi_q^{\lambda}}(X_I)&=(-1)^{\bar I}\sum_{ \tau \in \Sym_{I}, \atop \mu \in \mathcal M(\Sym_r/\Sym_I) } \langle i_1,\dots ,i_r\mid  \check{R}_{\tau}\check{R}_{\mu} \mathcal{E}_{\Tc}X_1\cdots X_{r} \check{R}_{\mu^{-1}} \check{R}_{\tau^{-1}}\mid i_1,\dots ,i_r\rangle\\
  &=\frac{(-1)^{\bar I}\alpha_{\mathrm{q}^2}(I) }{\alpha( {I})}\sum_{\si\in \Sym_r}\langle i_{\sigma_1},\dots ,i_{\sigma_r}\mid  \mathcal{E}_{\Tc}X_1\cdots X_{r} \mid i_{\sigma_1},\dots ,i_{\sigma_r}\rangle.
\eal
\end{equation*}

If $\lambda \notin H(m,n;r)$, so $\lambda_{m+1}>n$. Take the standard Young tableau $\Tc_0$ obtained by numbering the boxes by rows downwards, from left to right in every row.  Similar to the  supercommutative case in \cite{JLZ3}, it follows from the fusion procedure of primitive idempotents and lemma \ref{idem-lemma} that for any $\tau\in \Sym_r$
\beq
\mathcal{E}_{\Tc_0}^{\lambda}\cdot e_{i_{\tau(1)}}\ot \cdots \ot e_{i_{\tau(r)}}=0.
\eeq
Hence,
\beq
\bal
 \frac{\alpha({I})}{\alpha_{\mathrm{q}^2}(I) }\Imm_{\chi_q^{\lambda}}(X_I)&=(-1)^{\bar I}\sum_{\si\in \Sym_r}\langle i_{\sigma_1},\dots ,i_{\sigma_r}\mid  \mathcal{E}_{\Tc_0}^{\lambda} X_1\cdots X_{r} \mid i_{\sigma_1},\dots ,i_{\sigma_r}\rangle\\
&=(-1)^{\bar I}\sum_{\si\in \Sym_r}(\mathcal{E}_{\Tc_0}^{\lambda} X_1\cdots X_{r})_{i_{\si_1},\ldots, i_{\si_r}}^{i_{\si_1},\ldots, i_{\si_r}}\\
&=(-1)^{\bar I}\sum_{\si,\tau\in \Sym_r} (\mathcal{E}_{\Tc_0}^{\lambda})_{i_{\tau_1},\ldots, i_{\tau_r}}^{i_{\si_1},\ldots, i_{\si_r}} x_{i_{\tau_1},i_{\si_1}} \cdots x_{i_{\tau_r},i_{\si_r}}
\eal
\eeq
which vanishes when $\lambda \notin H(m,n;r)$.

%For case $(i)$, it is same as  the proof  in the non-super case, see \cite{JLZ}. Next  we see the case $(ii)$, by lemma \ref{idem-lemma}
%\beq
%\bal
%&\mathcal{E}_{\Tc_0}^{\lambda}\cdot e_{i_1}\ot \cdots \ot e_{i_r}\\
%=&-q^{-1} \mathcal{E}_{\Tc_0}^{\lambda}\check{R}_{(k,k+1)}\cdot e_{i_1}\ot \cdots \ot e_{i_r}\\
%=&-\mathcal{E}_{\Tc_0}^{\lambda}\cdot e_{i_1}\ot \cdots \ot e_{i_r}.
%\eal
%\eeq
%Combining case $(i)$ and $(ii)$, $\mathcal{E}_{\Tc_0}^{\lambda}\cdot e_{i_1}\ot \cdots \ot e_{i_r}=0$.
%
%Moreover,
%by the same argument on the above proof, we can prove that for any $\Tc\in \mathrm{SYT}(\lambda)$,
%\beq
%\mathcal{E}_{\Tc}^{\lambda}\cdot e_{i_1}\ot \cdots \ot e_{i_r}=0.
%\eeq
%Note that for any $\tau \in \mathcal M(\Sym_r/\Sym_I)$
%\beq
%\bal
%&\mathcal{E}_{\Tc_0}^{\lambda}\cdot e_{i_{\tau(1)}}\ot \cdots \ot e_{i_{\tau(r)}}\\
%=&(-1)^{\bar I_{\tau^{-1}}}\mathcal{E}_{\Tc_0}^{\lambda}\check{R}_{\tau^{-1}}\cdot e_{i_1}\ot \cdots \ot e_{i_r}.
%\eal
%\eeq
%So by lemma \ref{idem-lemma}, the above  equation can be decompose with each term containing $\mathcal{E}_{\Tc}^{\lambda}\cdot e_{i_1}\ot \cdots \ot e_{i_r}$, for some $\Tc\in \mathrm{SYT}(\lambda)$, so it is zero. Hence we obtain that $\Imm_{\chi_q^{\lambda}}(X_I)=0$, when $\lambda \notin H(m,n;r)$.
\end{proof}

\subsection{Quantum super Schur-Weyl duality}
We review the quantum super version of Schur-Weyl duality introduced by Moon \cite{Moo} and by Mitsuhashi \cite{Mit}.

 The tensor space $(\CC^{m|n})^{\ot r}$ is a multiplicity free $U_q(\gl_{m|n})\times  \Hc_r $-module and there exists a decomposition of the space of tensors
\beq\label{Schur-Weyl-decom}
(\CC^{m|n})^{\ot r}\cong\bigoplus_{\lambda \in H(m,n;r)} U^{\lambda}\ot V^{\lambda},
\eeq
where $(\pi_{\lambda},U^{\lambda})$ are the irreducible covariant tensor representations for $U_q(\gl_{m|n})$ corresponding to  partition $\lambda$ and $(\xi_{\lambda},V^{\lambda})$ are the irreducible representations  of $\Hc_r$.
Hence there exists an integral dominant weight $\bm\lambda$ such that the highest weight module $L(\bm\lambda)$ is isomorphic to $U^{\lambda}$. Such a highest weight $\bm\lambda$ is called {\it covariant highest weight}.
The one to one correspondence between the covariant highest weights and partitions in $H(m,n)$ is given as follows.
For any partition $\lambda$ in $H(m,n)$, the corresponding covariant highest weight $\bm\lambda=(\bm\lambda_1,\ldots, \bm\lambda_m \mid \bm\lambda_{m+1},\ldots, \bm\lambda_{m+n})$ is given by
\beq
\bal
\bm\lambda_{i}=\lambda_i, \ 1 \leq i \leq m,\\
\bm\lambda_{m+i}=\max\{0,\lambda'_i-m\}, \ 1\leq i \leq n,
\eal
\eeq
where $\lambda'$ is the partition conjugate to $\lambda$.
%Conversely if $\bm\lambda=(\bm\lambda_1,\ldots, \bm\lambda_m \mid \bm\lambda_{m+1},\ldots, \bm\lambda_{m+n})$ is a covariant highest weight
% then the components of $\lambda$ are given explicitly by
%\beq
%\bal
%\lambda_i=\bm\lambda_i, \ 1\leq i \leq m,\\
%\lambda_{m+i}= \sharp \{j \mid \bm\lambda_{m+j}\geq i, \ 1\leq j \leq n\}, \ 1\leq i \leq n.
%\eal
%\eeq

Let $\{v_{\Tc} \mid \Tc \in \mathrm{SYT}(\lambda)\}$ be the Young's orthonormal
basis of $V^{\lambda}$.
Define a nondegenerate symmetric  bilinear form $\langle \cdot | \cdot \rangle$ on  $(\CC^{m|n})^{\ot r}$ by
\beq\label{inner-product}
\bal
\langle e_{i_1}\ot \cdots \ot e_{i_r}, e_{j_1}\ot \cdots \ot e_{j_r}  \rangle&=\prod_{1 \leq k \leq r}\delta_{i_k,j_k}.
\eal
\eeq
We introduce the (involutive) anti-automorphic $*$-operations for $\Hc_r $  and $U_q(\gl_{m|n})$  respectively. Let  $T_{\si}^{*}=T_{\si^{-1}} \in \Hc_r$. And $(q^h)^{*}=q^h, E_i^*=F_i,$ and  $F_i^*=E_i$ in $U_q(\gl_{m|n})$.
\begin{lemma}\label{adjoint-op}
The $*$-operations for $\Hc_r$  and $U_q(\gl_{m|n})$  afford a contravariant  bilinear form, i.e.
\beq\label{contravariant-act}
\bal
  &\langle x\cdot  e_{i_1}\ot \cdots \ot e_{i_r}, e_{j_1}\ot \cdots \ot e_{j_r}  \rangle\\
  &=\langle  e_{i_1}\ot \cdots \ot e_{i_r}, x^* \cdot e_{j_1}\ot \cdots \ot e_{j_r}  \rangle,
\eal
\eeq
where $x$ is any element in $\Hc_r$ or $U_q(\gl_{m|n})$.
\end{lemma}

There is a parameterization of basis vectors in $L(\bm \lambda)$ by the combinatorial objects called the {\it Gelfand-Tsetlin patterns} (GT patterns for short). Such a pattern  $\Lambda=(\lambda_{ij})$ (associated with $\bm \lambda$) is a triangular array
\begin{equation}\label{GT pattern}
%\Lambda=
\begin{array}{cccccccc}
   \lambda_{m+n,1}   &  \cdots      & \lambda_{m+n,m} & \lambda_{m+n,m+1} & \cdots & \lambda_{m+n,m+n-1} & \lambda_{m+n,m+n} \\
  \lambda_{m+n-1,1} &  \cdots     & \lambda_{m+n-1,m} & \lambda_{m+n-1,m+1} & \cdots & \lambda_{m+n-1,m+n-1} & \\
  \vdots    &  \vdots & \vdots & \vdots   & \reflectbox{$\ddots$}    \\
  \lambda_{m+1,1} &  \cdots  & \lambda_{m+1,m} & \lambda_{m+1,m+1}    \\
  \lambda_{m,1} &  \cdots  & \lambda_{m,m}    \\                       \vdots    &  \reflectbox{$\ddots$}    \\
  \lambda_{1,1}
\end{array}
\end{equation}
The set of  vectors $\zeta_{\Lambda}$ parameterized by all GT patterns  $\Lambda=(\lambda_{ij})$ satisfying the conditions
\begin{equation}\label{pattern-condition}
 \begin{array}{rl}
(1) & \lambda_{m+n,j}=\bm \lambda_j , \ 1\leq j\leq m+n;\\
(2)&\theta_{p-1,i}:=\lambda_{pi}-\lambda_{p-1,i}\in\{0,1\},\ 1\leq i\leq m, m+1\leq p\leq m+n;\\
(3)  &  \lambda_{pm}\geq \# \{i:\lambda_{pi}>0, m+1\leq i \leq p\}, \  m+1\leq p\leq m+n ;\\
(4)& \text{if }
\lambda_{m+1,m}=0, \text{then}\; \theta_{mm}=0; \\
(5)& \lambda_{pi}-\lambda_{p,i+1}\in{\mathbb Z}_+,\ 1\leq i\leq m-1,
    m+1\leq p\leq m+n-1;\\
(6)& \lambda_{i,j}-\lambda_{i-1,j}\in{\mathbb Z}_+\text{ and }\lambda_{i-1,j}-\lambda_{i,j+1}\in{\mathbb Z}_+,\\
  & 1\leq j\leq i\leq m\text{ or } m+1\leq j\leq i\leq m+n.
 \end{array}
\end{equation}
constitutes a basis in $L(\bm\lambda)$.

Given a GT pattern $\Lambda=(\lambda_{ij})$, we set
\begin{equation}\label{eq:lki}
l_{ij}=\left\{\begin{array}{cc}
\lambda_{ij}-j+1, &1 \leq j\leq m;\\
-\lambda_{ij}+j-2m, &m+1\leq j\
\leq  m+n.
\end{array}\right.
\end{equation}

\begin{proposition}[\cite{PSV, Lu}]
The actions of the generators of $U_q(\gl_{m|n})$ are given by the formulas
\begin{equation*}
q^{\epsilon_k}\zeta_{\Lambda}=q^{\left(\sum_{j=1}^{k}\lambda_{kj}-\sum_{j=1}^{k-1}\lambda_{k-1,j}\right)}\zeta_{\Lambda},\quad 1\leq k\leq m+n;
\end{equation*}
\begin{equation*}
E_k\zeta_{\Lambda}=-\sum_{i=1}^{k}\frac{\Pi_{j=1}^{k+1}(l_{k+1,j}-l_{ki}) }
  {\Pi_{j\neq i,j=1}^{k} (l_{kj}-l_{ki}) }\zeta_{\Lambda+\delta_{ki}},
\quad 1\leq k\leq m-1;
\end{equation*}
\begin{equation*}
F_k\zeta_{\Lambda}=\sum_{i=1}^{k}\frac{\Pi_{j=1}^{k-1}(l_{k-1,j}-l_{ki}) }{\Pi_{j\neq i,j=1}^{k} (l_{kj}-l_{ki}) }\zeta_{\Lambda-\delta_{ki}},
\quad 1\leq k\leq m-1;
\end{equation*}
\ben
\bal
E_m\zeta_{\Lambda}=&\sum_{i=1}^{m}\theta_{mi}(-1)^{i-1}(-1)^{\theta_{m1}+\ldots+\theta_{m,i-1}}\\
&\times  \frac{\prod_{1\leq j< i\leq m} (l_{mj}-l_{mi}-1)}
  {\prod_{1\leq i<j\leq m} (l_{mj}-l_{mi})
    \prod_{j\neq i,j=1}^{m}(l_{m+1,j}-l_{mi}-1)}
    \zeta_{\Lambda+\delta_{mi}},
\eal
\een

\begin{equation*}
\begin{split}
F_m \zeta_{\Lambda}=&\sum_{i=1}^{m}(1-\theta_{mi})(-1)^{i-1}(-1)^{\theta_{m1}+\ldots+\theta_{m,i-1}}
\\
&\times  \frac{(l_{mi}-l_{m+1,m+1})\Pi_{1\leq  i<j\leq m} (l_{mj}-l_{mi}+1)\Pi_{j=1}^{m-1}(l_{m-1,j}-l_{mi})}
  {\Pi_{1\leq j< i\leq m} (l_{mj}-l_{mi})} \zeta_{\Lambda-\delta_{mi}},
\end{split}
\end{equation*}
and for $m+1\leq k\leq m+n-1$
\begin{equation*}
\begin{split}
E_k \zeta_{\Lambda}=&\sum_{i=1}^{m}\theta_{ki}(-1)^{\vartheta_{ki}}(1-\theta_{k-1,i})\times
\prod_{j\neq i,j =1}^{m}\frac{(l_{kj}-l_{ki}-1)}{(l_{k+1,j}-l_{ki}-1)}
\zeta_{\Lambda+\delta_{ki}}
\\
&-\sum_{i=m+1}^{k}
\prod_{j=1}^{m}\frac{(l_{kj}-l_{ki})(l_{kj}-l_{ki}+1)}{(l_{k+1,j}-l_{ki})(l_{k-1,j}-l_{ki}+1)} \times
  \frac{\Pi_{j=m+1}^{k+1}(l_{k+1,j}-l_{ki})}
  {\Pi_{j\neq i,j=m+1}^{k} (l_{kj}-l_{ki})}\zeta_{\Lambda+\delta_{ki}},
\end{split}
\end{equation*}

\begin{equation*}
\begin{split}
F_k\zeta_{\Lambda}=&
\sum_{i=1}^{m}\theta_{k-1,i}(-1)^{\vartheta_{ki}}(1-\theta_{ki})\times\frac{\Pi_{j=m+1}^{k+1}(l_{k+1,j}-l_{ki})\Pi_{j=m+1}^{k-1}(l_{k-1,j}-l_{ki}+1)}
  {\Pi_{j=m+1}^{k} (l_{kj}-l_{ki})(l_{kj}-l_{ki}+1)}
\\
&\times
\prod_{j\neq i,j=1}^{m}\frac{(l_{kj}-l_{ki}+1)}{(l_{k-1,j}-l_{ki}+1)}
\zeta_{\Lambda-\delta_{ki}}
  + \sum_{i=m+1}^{k}
\frac{\prod_{j=m+1}^{k-1}(l_{k-1,j}-l_{ki})}{\prod_{j\neq i,j =m+1}^{k}(l_{kj}-l_{ki})}
\zeta_{\Lambda-\delta_{ki}}.
\end{split}
\end{equation*}
Here $\vartheta_{k,i}=\theta_{k1}+\ldots+\theta_{k,i-1}+\theta_{k-1,i+1}+\ldots+\theta_{k-1,m}$. The patterns $\Lambda\pm \delta_{ki}$ are obtained from $\Lambda$ by replacing $\lambda_{ki}$ with $\lambda_{ki}\pm1$.  We assume  that
$\zeta_{\Lambda}=0$  if the pattern $\Lambda$ does not satisfy the conditions \eqref{pattern-condition}.
\end{proposition}

We can identity the GT patterns $\Lambda=(\lambda_{ij})$ with the semistandard supertableaux  of shape $\lambda \in H(m,n;r)$, which satisfying the following conditions:

(a) the entries weakly increase from left to right along each row and down each column;

(b) the entries in $\{1,\ldots, m\}$ strictly increase down each column;

(c) the entries in $\{m + 1,\ldots, m+n\}$ strictly increase from left to right along each row.\\
We denote by $SSYT(\lambda)$ the set of semistandard supertableaux of shape $\lambda$. We say a supertableau $\Lambda$ of weight $\mu$ is that  $\mu_i$ equals the number of $i's$ in $\Lambda$.

Define the matrix with spectral parameter $u$
\beq
T(u)=u T^{+}-u^{-1}T^{-}.
\eeq
The {\it quantum Berezinian} is an element in $U_q(\gl_{m|n})[[u]]$ defined by
\begin{equation}
\begin{split}
\Ber_{m+n}(u)=&\sum_{\si\in \Sym_{m}} (-q)^{-l(\si)} T(uq^{-m+1})_{\si_m,m}\cdots T(u)_{\si_1,1}
\\
&\times\sum_{\tau\in \Sym_{n}}
 (-q)^{-l(\tau)} T(uq^{-m+1})_{m+1,m+\tau(1)}^{-1} \cdots  T(uq^{-m+n})_{m+n,m+\tau(n)}^{-1}.
\end{split}
\end{equation}
The coefficients of quantum Berezinian are central in $U_q(\gl_{m|n})$ proved in  \cite{JLiZ} via an evaluation homomorphism from quantum affine superalgebra $U_q(\widehat{\gl}_{m|n})$ to $U_q(\gl_{m|n})$.
Then $\Ber_{m+n}(u)$  acts on $L(\bm \lambda)$ by the scalar
\begin{equation}
\frac{\prod_{i=1}^{m}(uq^{l_{m+n,i}}-u^{-1}q^{-l_{m+n,i}})}
{\prod_{j=m+1}^{m+n}(uq^{l_{m+n,j}}-u^{-1}q^{-l_{m+n,j}})}.
\end{equation}
From the branching rules of $L(\bm\lambda)$, we can obtain that
\beq\label{central-char}
\bal
 &\Ber_k(u)\zeta_{\Lambda}=\prod_{i=1}^{k}(uq^{l_{ki}}-u^{-1}q^{-l_{ki}}) \zeta_{\Lambda} \ \ \text{for} \ 1 \leq k \leq m,\\
 &\Ber_k(u)\zeta_{\Lambda}=\frac{\prod_{i=1}^{m}(uq^{l_{ki}}-u^{-1}q^{-l_{ki}})}
{\prod_{j=m+1}^{k}(uq^{l_{kj}}-u^{-1}q^{-l_{kj}})} \zeta_{\Lambda} \ \ \text{for} \ m+1 \leq k \leq m+n.
\eal
\eeq

\subsection{Weight space interpretation of quantum super immanants}
We can decompose \eqref{Schur-Weyl-decom} into basis vectors for $U^{\lambda}\ot V^{\lambda}$
\beq\label{Schur-Weyl-decom2}
(\CC^{m|n})^{\ot r}\cong\sum_{\lambda \in H(m,n;r)} \sum_{(\Lambda,\Tc)} \mathbb C \zeta_{\Lambda}\ot v_{\Tc}.
\eeq
 By Lemma \ref{adjoint-op}, these basis vectors in decomposition \eqref{Schur-Weyl-decom2} are normalized orthogonal, i.e.
$$\langle \zeta_{\Lambda}\ot v_{\Tc}, \zeta_{\Lambda'}\ot v_{\Tc'}\rangle =\delta_{\Lambda\Lambda'}\delta_{\Tc\Tc'}$$
for any semistandard supertableaux $\Lambda,\Lambda'$ and standard  Young tableaux $\Tc,\Tc'$.

An $(m+n)$-tuple $(a_1,\ldots,a_{m+n})$ in $\mathbb{Z}_{\geq 0}^{m+n}$ such that
$a_1+\cdots +a_{m+n}=r$ is called {\it a weak composition of $r$ into $(m+n)$ parts} or {\it a weak $(m+n)$-composition of $r$}.
It's clear that
the weights of $U_q(\gl_{m|n})$ in $(\CC^{m|n})^{\ot r}$ are exactly weak $(m+n)$-compositions of $r$.
For any weak $(m+n)$-composition $\bm\mu$ of $r$, we define the projection operator
$$\mathcal P_{\bm\mu}: (\CC^{m|n})^{\ot r}\rightarrow ((\CC^{m|n})^{\ot r})_{\bm \mu},$$
then $\mathcal P_{\bm \mu}: L(\bm\lambda)\rightarrow L(\bm\lambda)_{\bm \mu}$, where $( (\CC^{m|n})^{\ot r})_{\bm \mu}$ and $L(\bm\lambda)_{\bm \mu}$ are the subspaces with weight $\bm \mu$ as  $U_q(\gl_{m|n})$-modules.

Now we give a  weight space interpretation of quantum super immanants.
\begin{theorem}\label{q-super-Kostant-thm}
  Let $\lambda \in H(m,n;r)$, $\bm\mu$ be  a weak  $(m+n)$-composition of $r$, and assume the multiset $I=(1^{\mu_1},2^{\mu_2},\ldots,(m+n)^{\mu_{m+n}})$ of $[m+n]$. Then
\beq\label{q-super-Kostant-iden}
\frac{\Imm_{\chi_q^{\lambda}}(X_I)}{\alpha_{\mathrm{q}^2}(I) }=str(\mathcal P_{\bm\mu}\ot 1)\circ\Pi_r\circ \mathcal P_{\bm \mu}|_{L(\bm\lambda)}.
\eeq
\end{theorem}

Given any weight $\bm\mu=(\bm\mu_1,\ldots,\bm\mu_{m+n})$, which can be written as a multiset $I=(1^{\mu_1} ,2^{\mu_2} ,\ldots ,(m+n)^{\mu_{m+n}})=(i_1, \ldots, i_r)$. We define  the following map
\beq
\bal
\theta_{\bm\mu}: SYT(\lambda) \rightarrow YT(\lambda)\\
\eal
\eeq
which maps any standard Young tableau $\Tc$ to a Young tableau $\theta_{\bm\mu}(\Tc)$ that replaces each node $k$ in $\Tc$  by $i_k$ for $1 \leq k \leq r$. It is notable that $SSYT(\lambda)$ is contained in $YT(\lambda)$.

The theorem \ref{q-super-Kostant-thm} can follow from the following lemma, which gives
an explicit correspondence between the Gelfand-Tsetlin type bases of the  covariant representations of $U_q(\gl_{m|n})$ and
Young's orthonormal basis of  irreducible representations of $\Hc_r$.
\begin{lemma}\label{schur-weyl-corr}
Let $\lambda \vdash r$ and $\bm\mu$ be a weak $(m+n)$-composition of $r$ written as the multiset  $I=(1^{\mu_1},2^{\mu_2},\ldots,(m+n)^{\mu_{m+n}})=(i_1 \leq \ldots \leq i_r)$ of $[m+n]$, let  $\mathcal{E}_{\Tc}^{\lambda}$ be the primitive idempotent   of $\Hc_r$ associated with ${\Tc}$. Then as a $ U_q(\mathfrak{gl}_{m|n}) \times \Hc_r $-module
\begin{align}\label{schur-weyl-iden1}
\sqrt{\frac{c_{\lambda}}{\alpha_{\mathrm{q}^2}(I) }}\mathcal{E}_{\Tc}^{\lambda}\cdot e_{i_1}\ot \cdots \ot e_{i_r}=\left\{\begin{array}{cc}
c\cdot \zeta_{\theta_{\bm\mu}(\Tc)}\ot v_{\Tc}&\text{ if } \theta_{\bm\mu}(\Tc)  \in SSYT(\lambda),\\
0&\text{ if }\theta_{\bm\mu}(\Tc)\notin  SSYT(\lambda),
\end{array}\right.
\end{align}
where $c_{\lambda}$  is the Schur element and $c$ is a nonzero constant. Explicitly,
$c=1$, if $\Tc$ is the unique pre-image of $\theta_{\bm\mu}(\Tc)$. And
$\sum_{i=1}^s||c_i||^2=1$ for $\theta_{\bm\mu}(\Tc_1)=\cdots=\theta_{\bm\mu}(\Tc_s)$ $(s>1)$.
In particular,
assume $r \leq m+n$, then
\beq\label{schur-weyl-iden2}
\sqrt{c_{\lambda}} \mathcal{E}_{\Tc}^{\lambda}\cdot e_1\ot \cdots \ot e_r=\zeta_{\Tc}\ot v_{\Tc} .
\eeq
\end{lemma}

\begin{proof}[\textbf{Proof of theorem \ref{q-super-Kostant-thm}}]
By the definition of quantum super immanants,
the left side of \eqref{q-super-Kostant-iden} equals to
\beq
\bal
 \frac{\Imm_{\chi_q^{\lambda}}(X_I)}{\alpha_{\mathrm{q}^2}(I) }&=\frac{(-1)^{\bar I}}{\alpha_{\mathrm{q}^2}(I) }\sum_{\si\in \Sym_r}\langle i_{1},\dots ,i_{r}\mid  \chi_q^{\lambda} X_1\cdots X_{r} \mid i_{1},\dots ,i_{r}\rangle\\
&=\frac{(-1)^{\bar I}}{\alpha_{\mathrm{q}^2}(I) }(\chi_q^{\lambda}X_1\cdots X_{r})_{i_{1},\ldots,i_{r}}^{i_{1},\ldots,i_{r}}\\
&=\frac{(-1)^{\bar I}}{\alpha_{\mathrm{q}^2}(I) }\sum_{\tau\in \Sym_r} \chi_q^{\lambda}(\Hc_{\tau}) (\check{R}_{\tau})_{i_{\tau_1},\ldots,i_{\tau_r}}^{i_{1},\ldots,i_{r}} x_{i_{\tau_1},i_1} \cdots x_{i_{\tau_r},i_r}.
\eal
\eeq

According to Lemma \ref{schur-weyl-corr},   the right side of \eqref{q-super-Kostant-iden} equals to
\begin{align*}
&str(\mathcal P_{\bm\mu}\ot 1)\circ\Pi_r\circ \mathcal P_{\bm\mu}|_{L(\bm\lambda)}\\
&=(-1)^{\bar I}\langle  \sqrt{\frac{c_{\lambda}}{\alpha_{\mathrm{q}^2}(I) }}(\mathcal P_{\bm\mu}\ot 1) \circ \Pi_r \circ\sum_{\Tc}\mathcal{E}_{\Tc}\cdot e_{i_1}\ot \cdots \ot e_{i_r}, \\
& \quad \sqrt{\frac{c_{\lambda}}{\alpha_{\mathrm{q}^2}(I) }}\sum_{\Tc}\mathcal{E}_{\Tc}\cdot e_{i_1}\ot \cdots \ot e_{i_r}\rangle\\
&=\frac{(-1)^{\bar I}c_{\lambda}}{\alpha_{\mathrm{q}^2}(I)}\langle  (\mathcal P_{\bm\mu}\ot 1) \circ \Pi_r \circ\sum_{\Tc}\mathcal{E}_{\Tc}\cdot e_{i_1}\ot \cdots \ot e_{i_r}, e_{i_1}\ot \cdots \ot e_{i_r}\rangle.
\end{align*}

By the relation of characters and primitive idempotents, one has that
\beq
\bal
&\frac{(-1)^{\bar I}c_{\lambda}}{\alpha_{\mathrm{q}^2}(I)} (\mathcal P_{\bm\mu}\ot 1) \circ  \Pi_r \circ\sum_{\Tc}\mathcal{E}_{\Tc}\cdot e_{i_1}\ot \cdots \ot e_{i_r}\\
&=\frac{(-1)^{\bar I}}{\alpha_{\mathrm{q}^2}(I)} (\mathcal P_{\bm\mu}\ot 1)\circ \Pi_r \circ\sum_{\si\in \Sym_r} \chi_q^{\lambda}(\Hc_{\si}) \check{R}_{\si} \cdot e_{i_1}\ot \cdots \ot e_{i_r}\\
&=\frac{(-1)^{\bar I}}{\alpha_{\mathrm{q}^2}(I)} \sum_{\si\in \Sym_r} \chi_q^{\lambda}(\Hc_{\si}) \check{R}_{\si}\circ (\mathcal P_{\bm\mu}\ot 1)\sum_{(j_1,\ldots,j_r)} (-1)^{\sum\limits_{a<b}\bar j_a(\bar i_b+\bar j_b)} x_{j_1,i_1}\cdots x_{j_r,i_r}\\
&\quad \ot e_{j_1}\ot \cdots \ot e_{j_r}\\
&=\frac{(-1)^{\bar I}}{\alpha_{\mathrm{q}^2}(I)} \sum_{\si,\tau\in \Sym_r} (-1)^{\sum\limits_{a<b}\bar i_{\tau(a)}(\bar i_b+\bar i_{\tau(b)})} \chi_q^{\lambda}(\Hc_{\si})  x_{i_{\tau(1)},i_1}\cdots x_{i_{\tau(r)},i_r}\\
&\quad \ot \check{R}_{\si}\cdot e_{i_{\tau(1)}}\ot \cdots \ot e_{i_{\tau(r)}}.
\eal
\eeq
So taking the  bilinear form with $e_{i_1}\ot \cdots \ot e_{i_r}$, these terms in the summation of the above equation are zero unless $\si=\tau$. Then
%\beq
%\bal
%&str(\mathcal P_{\bm\mu}\ot 1)\circ \Pi_r\circ \mathcal P_{\bm\mu}|_{L(\bm\lambda)}\\
%&=\frac{(-1)^{\bar I}}{\alpha_{\widetilde{q}_I^2}(I)}  \sum_{\tau\in \Sym_r} (-1)^{\overline{\tau^{-1}(I)}_{\tau} +\sum\limits_{a<b}\bar i_{\tau(a)}(\bar i_b+\bar i_{\tau(b)})} \chi_q^{\lambda}(\Hc_{\tau})  x_{i_{\tau(1)},i_1}\cdots x_{i_{\tau(r)},i_r}.
%\eal
%\eeq
%Since
%\ben
%\bal
%& P_{\tau}\cdot e_{i_{\tau(1)}}\ot \cdots \ot e_{i_{\tau(r)}}\\
%&=(-1)^{\overline{\tau^{-1}(I)}_{\tau}}e_{i_1}\ot \cdots \ot e_{i_r}\\
%&=(-1)^{\sum\limits_{a<b}\bar i_{\tau(a)}(\bar i_b+\bar i_{\tau(b)})}(P_{\tau})_{i_{\tau(1)},\ldots,i_{\tau(r)}}^{i_{1},\ldots,i_{r}}
%e_{i_1}\ot \cdots \ot e_{i_r},
%\eal
%\een
%we obtain that
\beq
\bal
&str(\mathcal P_{\bm\mu}\ot 1)\circ\Pi_r\circ \mathcal P_{\bm\mu}|_{L(\bm\lambda)}\\
&=\frac{(-1)^{\bar I}}{\alpha_{\mathrm{q}^2}(I)} \sum_{\tau\in \Sym_r} \chi_q^{\lambda}(\Hc_{\tau}) (\check{R}_{\tau})_{i_{\tau_1},\ldots,i_{\tau_r}}^{i_{1},\ldots,i_{r}} x_{i_{\tau_1},i_1} \cdots x_{i_{\tau_r},i_r}\\
&=\frac{\Imm_{\chi_q^{\lambda}}(X_I)}{\alpha_{\mathrm{q}^2}(I)}.
\eal
\eeq
\end{proof}

\section{Quantum super Littlewood correspondences}
We establish the Littlewood correspondences between quantum super immanants and  Schur supersymmetric polynomials. As special cases, we obtain the super case $(q\rightarrow 1)$, and the quantum case $(n=0)$.

\subsection{Quantum super Littlewood correspondences I and II}
We have a quantum super version of the Littlewood correspondences I and II \cite{Li}[p. 118, p. 120].
\begin{theorem}\label{LittlewoodI}
   Corresponding to any relation between Schur supersymmetric polynomials of total order $m+n$, we may replace each Schur supersymmetric polynomial by the corresponding quantum super immanant of complementary principal minors of generator matrix of $A_q(\Mat_{m|n})$ provided that every product is summed for all sequences of pairwise disjoint subsets of $[m+n]$.
\end{theorem}
In general,
\begin{theorem}\label{LittlewoodII}
   Corresponding to any relation between Schur supersymmetric polynomials, we may replace each Schur supersymmetric polynomial by the corresponding (normalized) quantum super immanant of a principal minor of generator matrix of $A_q(\Mat_{m|n})$ provided that we sum for all principal minors of the appropriate order with non-decreasing ordered multisets of $[m+n]$.
\end{theorem}
We only prove Theorem \ref{LittlewoodII}, Theorem \ref{LittlewoodI} can be viewed as a special case of  Theorem \ref{LittlewoodII}.
\begin{proof}[Proof of Theorem \ref{LittlewoodII} ]
 We consider the induced representation of tensor product of any irreducible representations $V^{\mu}$ and $V^{\nu}$, suppose that $|\mu|+|\nu|=k$, the character can be decomposed as
\begin{align*}
\chi_q^{\mathrm{Ind}(V^{\mu}\ot V^{\nu})}=\sum_{\si \in \mathfrak{S}_{k}}\check{R}_{\si}\mathcal{E}^{\mu}\mathcal{E}^{\nu}\check{R}_{\si^{-1}}=\sum_{\lambda\in H(m,n;k)}c_{\mu\nu}^{\lambda}\chi_q^{\lambda},
\end{align*}
where these coefficients $c_{\mu\nu}^{\lambda}$ are the Littlewood-Richardson coefficients.
Hence it follows from proposition \ref{imm-char} that  for any non-decreasing ordered multisets $I=(1\leq i_1\leq \ldots \leq i_{k} \leq m+n)$
\begin{equation}
\begin{aligned}
&\Imm_{\chi_q^{\mathrm{Ind}(V^{\mu}\ot V^{\nu})}}(X_I) \\
   &=(-1)^{\bar I}\langle i_1,\ldots ,i_k
\mid \chi_q^{\mathrm{Ind} (V^{\mu}\ot V^{\nu})} X_1\cdots X_{k} \mid i_1,\ldots ,i_k\rangle\\
   &=\frac{(-1)^{\bar I}\alpha_{\mathrm q^2}(I)}{\alpha( {I})}\sum_{\sigma\in\mathfrak{S}_k}\langle i_{\sigma_1},\dots ,i_{\sigma_k}\mid \mathcal{E}^{\mu} \mathcal{E}^{\nu}X_1\cdots X_k \mid i_{\sigma_1},\dots,i_{\sigma_k}\rangle.
\end{aligned}
\end{equation}

For each sequence $(I_1,I_2)$ of  non-decreasing ordered multisets of $[m+n]$ satisfying $|I_1|=|\mu|, \ |I_2|=|\nu|$ and the disjoint union of $I_1$ and $I_2$ is $I$,  there exists the minimal length coset representative elements $\si \in \mathcal M(\mathfrak{S}_k/\Sym_{|\mu|}\times \Sym_{|\nu|})$ such that $\si^{-1}(I)= (I_1,I_2)$.
Hence by proposition  \ref{imm-char}  we have that
\begin{equation*}
\begin{aligned}
&\Imm_{\chi_q^{\mathrm{Ind}(V^{\mu}\ot V^{\nu})}}(X_I) \\
&=\frac{(-1)^{\bar I}\alpha_{\mathrm q^2}(I)}{\alpha( {I})}\sum_{ \tau \in \Sym_{|\mu|}\times \Sym_{|\nu|}, \atop \si \in \mathcal M(\mathfrak{S}_k/\Sym_{|\mu|}\times \Sym_{|\nu|}) }\langle i_{\tau_{\si_1}},\dots ,i_{\tau_{\si_k}}\mid \mathcal{E}^{\mu} \mathcal{E}^{\nu} X_1 \cdots X_k \mid i_{\tau_{\si_1}},\dots ,i_{\tau_{\si_k}}\rangle\\
&=\sum_{(I_1,I_2)}\frac{\alpha_{\mathrm q^2}(I)}{\alpha_{\mathrm q^2}(I_1)\alpha_{\mathrm q^2}(I_2)}{\Imm}_{\chi^{\mu}}(X_{I_1}) {\Imm}_{\chi^{\nu}}(X_{I_2}),
\end{aligned}
\end{equation*}
where the sums are over all sequences $(I_1,I_2)$ of  non-decreasing ordered multisets of $[m+n]$  and the disjoint union of $I_j, \ j=1,2$ is $I$.

On the other hand,
\begin{equation}
\begin{aligned}
\Imm_{\chi_q^{\mathrm{Ind}(V^{\mu}\ot V^{\nu})}}(X_I)
   &=(-1)^{\bar I}\langle i_1,\ldots ,i_k
\mid \chi^{\mathrm{Ind}(V^{\mu}\ot V^{\nu})} X_1\cdots X_{k} \mid i_1,\ldots ,i_k \rangle\\
   &=\sum_{\lambda\in H(m,n;k)}c_{\mu\nu}^{\lambda}(-1)^{\bar I}\langle i_1,\ldots ,i_k  \mid \chi^{\lambda} X_1\cdots X_{k} \mid i_1,\ldots ,i_k \rangle\\
   &=\sum_{\lambda\in H(m,n;k)}c_{\mu\nu}^{\lambda}\Imm_{\chi^{\lambda}}(X_I).
\end{aligned}
\end{equation}
Hence,
\[
\sum_{(I_1,I_2)}\frac{\Imm_{\chi_q^{\mu}}(X_{I_1}) \Imm_{\chi_q^{\nu}}(X_{I_2})}{\alpha_{\mathrm q^2}(I_1)\alpha_{\mathrm q^2}(I_2)}  =\sum_{\lambda\in H(m,n;k)}c_{\mu\nu}^{\lambda}\frac{ \Imm_{\chi_q^{\lambda}}(X_I)}{\alpha_{\mathrm q^2}(I)}.
\]
This relation coincides with the product of the two Schur supersymmetric polynomials $\mathbb{S}_{\mu}$ and $\mathbb{S}_{\nu}$.
\end{proof}

In particular, we have general Littlewood-Merris-Watkins identities \cite{KS,MW} for the quantum supermatrices.
Let  $\psi_q^{\mu}$ (resp. $\phi_q^{\mu}$) be the induced characters of the sign character (resp. trivial character) of the parabolic subgroup $\Hc_{\mu}$ of $\Hc_k$.
%They can be decomposeed into  the irreducible $\Hc_k$-characters:
%\begin{align*}
% \psi^{\mu}=\sum_{\lambda}K_{\lambda^{T},\mu}\chi^{\lambda}, \ \ \
%\phi^{\mu}=\sum_{\lambda}K_{\lambda,\mu}\chi^{\lambda},
%\end{align*}
%where $K_{\lambda,\mu}$ are the Kostka numbers.
\begin{corollary}
   Let $I=(1\leq i_1\leq \ldots \leq i_k \leq m+n)$ be a multiset of $[m+n]$. Fix a partition $\lambda=(\lambda_1,\dots,\lambda_l)$ of $k$. Then
\begin{equation}\label{gene-qLMW}
\begin{aligned}
   &\Imm_{\psi_q^{\lambda}}(X_I)= \sum_{(I_1,\dots,I_l)}\frac{ \alpha_{\mathrm{q}^2}(I)}{ \alpha_{\mathrm{q}^2}(I_1)\cdots \alpha_{\mathrm{q}^2}(I_l)}\Imm_{\chi_q^{(1^{\lambda_1})}}(X_{I_1})\cdots \Imm_{\chi_q^{(1^{\lambda_l})}}(X_{I_l}),\\
   &\Imm_{\phi_q^{\lambda}}(X_I)=\sum_{(I_1,\dots,I_l)}\frac{ \alpha_{\mathrm{q}^2}(I)}{ \alpha_{\mathrm{q}^2}(I_1)\cdots \alpha_{\mathrm{q}^2}(I_l)}\Imm_{\chi_q^{(\lambda_1)}}(X_{I_1})\cdots \Imm_{\chi_q^{(\lambda_l)}}(X_{I_l}),
\end{aligned}
\end{equation}
where the sums are taken over all sequences $(I_1,\dots,I_l)$ of non-decreasing ordered multisets of $[m+n]$ satisfying $|I_j|=\lambda_j$ and the disjoint union of $I_j, 1 \leq j \leq l$ is $I$.
\end{corollary}

\subsection{MacMahon Master Theorem}
Define $\alpha_0=\beta_0=1$, for $k<0$ $\alpha_k=\beta_k=0$ and for $k>0$,
\begin{equation}\label{genefcn1}
\begin{aligned}
  &\alpha_k=str_{1,\ldots,k}\Ec^{(1^{k})}X_1\cdots X_k=\sum_{|I|=k}\frac{\Imm_{\chi_q^{(1^{k})}}(X_I)}{\alpha_{\mathrm{q}^2}(I)},\\ &\beta_k=str_{1,\ldots,k}\Ec^{(k)}X_1\cdots X_k=\sum_{|I|=k}\frac{\Imm_{\chi_q^{(k)}}(X_I)}{\alpha_{\mathrm{q}^2}(I)},
\end{aligned}
\end{equation}
where the sums are taken over all non-decreasing multisets $I$ with $k$ elements. Note that $\{\alpha_k \mid k \in \mathbb{Z}^{+}\}$ in $A_q(\Mat_{m|n})$  pairwise commute.
The  commutative subalgebra $\mathfrak{B}_{m|n}$ of $A_q(\Mat_{m|n})$ is generated by $\{\alpha_k,\mid  k \in \mathbb{Z}^{+}\}$.

We set
\begin{align}
\lambda(t)=\sum_{k=0}^{\infty}t^k \alpha_k,\qquad
\si(t)=\sum_{k=0}^{\infty}t^k \beta_k.
\end{align}
We have the quantum super analog of MacMahon Master Theorem in \cite{JLZ2}.
\begin{theorem}\label{Mac}
$\lambda(-t)\times \si(t) =1.$
\end{theorem}
%\begin{proof}
%It is sufficient to show that
%\begin{equation}\label{sum-term}
%\sum_{r=0}^k (-1)^r str_{1,\ldots,k}\Ec^{(r)}\ot \Ec^{(1^{k-r})}X_1\cdots X_k=0,
%\end{equation}
%where the idempotent $\Ec^{(1^{k-r})}$ is over the copies of $\mathrm{End}(\mathbb{C}^{m|n})$ labeled by $\{r+1,\ldots,k\}$.
%By the proof of Proposition \ref{imm-char}, we have that
%\ben
%\bal
%&str_{1,\ldots,k}\Ec^{(r)}\ot \Ec^{(1^{k-r})}X_1\cdots X_k\\
%=&\sum_{I}\frac{(-1)^{\bar I}} {\alpha(I)}\sum_{\si \in \mathfrak{S}_k}\langle i_{\sigma_1},\dots ,i_{\sigma_k}\mid \Ec^{(r)}\ot \Ec^{(1^{k-r})}X_1\cdots X_k \mid i_{\sigma_1},\dots ,i_{\sigma_k}\rangle\\
%=&\sum_{I}\frac{(-1)^{\bar I}} {\alpha(I)}\sum_{\si \in \mathfrak{S}_k}\langle i_1,\dots ,i_k\mid P_{\si}\Ec^{(r)}\ot \Ec^{(1^{k-r})}P_{\si^{-1}}X_1\cdots X_k \mid i_1,\dots ,i_k\rangle.
%\eal
%\een
%Since equation \eqref{ind-decom}, for $1 \leq r \leq k-1$
%\[
%\sum_{\si\in \Sym_k} P_{\si}\mathcal{E}^{(r)}\mathcal{E}^{(1^{\{r+1,\ldots,k\}})}P_{\si^{-1}}=\chi^{(r+1,1^{k-r-1})}+\chi^{(r,1^{k-r})}.
%\]
%Therefore the telescoping sum \eqref{sum-term} equals zero.
%\end{proof}

In the superalgebra $\End(\CC^{m|n})^{\ot 2}$, we define
\begin{equation}\label{P}
  P^q =\sum_{i,j=1}^{m+n}q^{\ve(i-j)}(-1)^{\bar{j}} e_{ij}\ot e_{ji}  .
\end{equation}
where  the symbol $\ve$  is defined by
\ben
\ve (i-j)
= \left\{ \begin{aligned}
1\ \ &\text{ if }\ i>j,\\
0\ \ &\text{ if }\ i=j,\\
-1\ \ &\text{ if }\ i<j.\\
\end{aligned} \right.
\een
Let $Y, Z\in \Mat_{m|n}(R)$. Denote $Y*Z=str_1 P^q Y_1Z_2$.
Let $X^{[k]}$ be the $k$th power of $X$ under the multiplication $*$, i.e.
\begin{align}\label{e:power-m}
X^{[0]}=1,\   X^{[1]}=X, \ X^{[k]}=X^{[k-1]}*X,\ k>1.
\end{align}
We denote
\begin{align}
\psi(t)=\sum_{k=0}^{\infty}t^k\gamma_{k+1},
\end{align}
where $\gamma_k=strX^{[k]}$.

The following Newton identities follow from the MacMahon Master Theorem, see \cite{JLZ2}.
\begin{theorem}[Newton's identities]\label{Newton-iden}
\begin{align}
&\partial_t \lambda(-t)=-\lambda(-t) \psi(t),\\
&\partial_t \si(t)= \psi(t)\si(t).
 \end{align}
\end{theorem}

\subsection{Quantum super Goulden-Jackson identities and Littlewood correspondence III }
Fix a partition $\lambda \vdash r$. On the  commutative algebra $\mathfrak{B}_{m|n}$ generated by $\alpha_i$ or $\beta_i$ \eqref{genefcn1}, we denote
\[
A=(\alpha_{\lambda_i^{T}-i+j})_{\lambda_1\times \lambda_1},\ \ \ \ \  B=(\beta_{\lambda_i-i+j})_{\lambda_1^T\times \lambda_1^T}.
\]
Then we have the following two special elements in $\mathfrak{B}_{m|n}$:
\begin{align*}
  &\det(A)=\sum_{\mu}K_{\lambda^{T},\mu}^{-1}\alpha_{\mu_1}\cdots \alpha_{\mu_{\lambda_1}}, \\
  &\det(B)=\sum_{\mu}K_{\lambda,\mu}^{-1}\beta_{\mu_1}\cdots \beta_{\mu_{\lambda_1}}.
\end{align*}

The following is a generalization of the Goulden-Jackson identities \cite{GJ,KS} in  quantum super setting.  It can be viewed as a generalization of the Jacobi-Trudi identity of Schur supersymmetric polynomials.
\begin{theorem}\label{quantum-JT} We have that
\begin{equation}\label{GJ-identity}
\det(A)=\det(B)=str_{1,\ldots,r}(\mathcal{E}_{\Tc}^{\lambda}X_1\cdots X_r) =\sum_{I}\frac{ \Imm_{\chi_q^{\lambda}}(X_I)}{\alpha_{\mathrm{q}^2}(I)},
\end{equation}
where the sum is over all non-decreasing ordered multisets $I$ of $[m+n]$ satisfying $|I|=r$.
\end{theorem}
\begin{proof}
  By definition and proposition \ref{imm-char},
\begin{equation*}
\begin{aligned}
\det(A)
 &=\sum_{\mu\vdash r} K_{\lambda^{T},\mu}^{-1}str_{1,\ldots,r}
 \mathcal{E}^{(1^{\mu_1})}\cdots \mathcal{E}^{(1^{\mu_{\lambda_1}})}X_1\cdots X_r\\
 &=\sum_{\mu\vdash r}\sum_{J}(-1)^{\bar J}\langle j_1,\dots ,j_r \mid K_{\lambda^{T},\mu}^{-1}\mathcal{E}^{(1^{\mu_1})}\cdots \mathcal{E}^{(1^{\mu_{\lambda_1}})}X_1\cdots X_r\mid  j_1,\dots ,j_r  \rangle,\\
 &=\sum_{\mu}\sum_{J'}\frac{(-1)^{\bar J'}} {\alpha_{\mathrm q^2}(J')}\langle  j'_1,\dots ,j'_r  \mid K_{\lambda^{T},\mu}^{-1} \sum_{\si\in \mathfrak{S}_r}\check{R}_{\si}\mathcal{E}^{(1^{\mu_1})}\cdots \\
 &\quad \cdots \mathcal{E}^{(1^{\mu_{\lambda_1}})}\check{R}_{\si^{-1}} X_1\cdots X_r\mid j'_1,\dots ,j'_r \rangle\\
 &=\sum_{\mu}\sum_{J'}\frac{(-1)^{\bar J'}}{\alpha_{\mathrm q^2}(J')}\langle j'_1,\dots ,j'_r \mid K_{\lambda^{T},\mu}^{-1} \psi_q^{\mu} X_1\cdots X_r\mid  j'_1,\dots ,j'_r \rangle\\
 &=\sum_{J'}\frac{(-1)^{\bar J'}}{\alpha_{\mathrm q^2}(J')}\langle  j'_1,\dots ,j'_r \mid\chi_q^{\lambda} X_1\cdots X_r\mid  j'_1,\dots ,j'_r\rangle\\
 &=str_{1,\ldots,r}(\mathcal{E}_{\Tc}^{\lambda}X_1\cdots X_r)=\sum_{J'}\frac{\Imm_{\chi_q^{\lambda}}(X_{J'})}{\alpha_{\mathrm q^2}(J')},
  \end{aligned}
  \end{equation*}
where the second sum in the second line is over all sequences $J$ with $r$ elements and the second sum of the third line runs over all non-decreasing multisets $J'=(j'_1\leq \dots \leq j'_r)$ of $[m+n]$, and the fifth line is because
   $\sum_{\mu}K_{\lambda^{T},\mu}^{-1}\psi_q^{\mu}=\chi_q^{\lambda}$. Similarly, it follows that $\det(B)=str_{1,\ldots,r}\mathcal{E}^{\lambda} X_1\cdots X_r$.
\end{proof}

Define a $q$-power series with coefficients in $A_q(\Mat_{m|n})$:
\begin{equation}\label{char}
 \Gamma(t)=\sum_{k=0}^{\infty}(-1)^k \alpha_{k}t^{m-n-k}.
\end{equation}
From the MacMahon Master Theorem,
\begin{equation}
  \Gamma(t)^{-1}=\sum_{k=0}^{\infty} \beta_{k}t^{n-m-k}.
\end{equation}
\begin{lemma}\label{w-solution}
There exists the unique elements $\omega_1,\ldots,\omega_m, \varpi_1,\ldots,\varpi_n$ over the algebraic closure field of fraction of integral domain generated by $\alpha_k, k>0$ satisfying the condition
\beq\label{char-equation}
\Gamma(t)\prod_{i=1}^n(t-\varpi_i)=\prod_{j=1}^m(t-\omega_j).
\eeq
\end{lemma}
\begin{proof}
We denote by $e_k$ (resp. $\bar{e}_k$) the $k$-th elementary symmetric polynomials in variables $\omega_j, 1 \leq j \leq m$ (resp. $\varpi_i, 1 \leq i \leq n$). Note that the leading component in the both sides of \eqref{char-equation} is $t^m$.
Then  compare the coefficients of $t^k\ ( -n \leq k \leq m-1)$ on the two sides of \eqref{char-equation}:
\beq\label{ek-equations}
\sum_{j=0}^{n}\alpha_{m-k-j}\bar e_j=\left\{\begin{array}{cc}
e_{m-k}, \ &0\leq k\leq m-1\\
0, \ &-n\leq k\leq -1\\
\end{array}\right.
\eeq
Note that each $e_i \ (1\leq i\leq m)$ is expressed as a polynomial in $\bar e_1,\ldots,\bar e_{n}$ from the m equations with  $0\leq k\leq m-1$. So we solve for elements  $\bar e_1,\ldots,\bar e_{n}$ firstly.
The coefficient matrix of the n equations with $-n \leq k\leq -1$ is
\[
A = \begin{pmatrix}
\alpha_m & \alpha_{m+1}& \cdots & \alpha_{m+n-1}\\
\alpha_{m-1} & \alpha_{m}& \cdots & \alpha_{m+n-2}\\
\vdots & \vdots & \ddots & \vdots\\
\alpha_{m+1-n} &\cdots & \cdots & \alpha_m
\end{pmatrix}_{n\times n}.
\]
Then $A=(\alpha_{\lambda_i^{T}-i+j})_{\lambda_1\times \lambda_1}$, where $\lambda=(n^m)$. By the $q$-super Goulden-Jackson identities \eqref{GJ-identity}, we have that $\det(A)\neq 0$. Hence there exists the unique solution $\bar e_1,\ldots,\bar e_{n}$ satisfying  the n equations with $-n \leq k\leq -1$. Then the the system of $(m+n)$ equations in variables  $ e_i, \bar e_j, 1 \leq i \leq m, 1\leq j  \leq n$ is solved. Hence these elements $\omega_1,\ldots,\omega_m, \varpi_1,\ldots,\varpi_n$ are determined uniquely by $\alpha_k$.

Moreover, considering any equation in the coefficients of  $t^k, k<-n$, we claim that this equation can be expressed as a linear combination of  the n equations with $-n \leq k\leq -1$ in \eqref{ek-equations} over the algebraic closure field. Indeed, the augmented matrix of the above $(n+1)$ equations  is
\[
A' = \begin{pmatrix}
\alpha_{m-k-n} & \alpha_{m-k-n+1}& \cdots & \alpha_{m-k}\\
\alpha_{m} & \alpha_{m+1}& \cdots & \alpha_{m+n}\\
\vdots & \vdots & \ddots & \vdots\\
\alpha_{m+1-n} &\cdots & \cdots & \alpha_{m+1}
\end{pmatrix}_{(n+1)\times(n+1)}.
\]
Then $A'=(\alpha_{\lambda_i^{T}-i+j})_{\lambda_1\times \lambda_1}$, where $\lambda=((n+1)^{m+1},\ldots)$. It follows from the $q$-super Goulden-Jackson identities \eqref{GJ-identity} and proposition \ref{imm-char} that $\det(A')= 0$.
\end{proof}
\begin{remark}
We conjecture that the following quantum super Cayley-Hamilton theorem holds
\begin{equation}\label{q-super-cay-ham}
\sum_{0\leq i \leq m,\atop 0\leq j \leq n} (-1)^{i+j}e_i(\omega_1,\ldots,\omega_m)X^{[m+n-i-j]}*e_j(\varpi_1,\dots,\varpi_n)I=0.
\end{equation}
For instance,
when $m=n=1$, we have that
\ben
\bal
&\omega_1=\alpha_1-\alpha_1^{-1}\alpha_2=x_{11}-qx_{12}x_{21}(x_{11}-x_{22})^{-1},\\
&\varpi_1=-\alpha_1^{-1}\alpha_2=x_{22}-qx_{12}x_{21}(x_{11}-x_{22})^{-1}.
\eal
\een
Then
\beq
(X-\omega_1I)*(X-\varpi_1I)=X^{[2]}-\omega_1X-X*\varpi_1I+\omega_1\varpi_1=0.
\eeq
%When $m=1,n=2$, we have that
%\ben
%\bal
%&e_1(\omega_1)=\omega_1=\alpha_1+(\alpha_1\alpha_2-\alpha_3)(\alpha_2-\alpha_1^2)^{-1},\\
%&e_1(\varpi_1,\varpi_2)=\varpi_1+\varpi_2=(\alpha_1\alpha_2-\alpha_3)(\alpha_2-\alpha_1^2)^{-1},\\
%&e_2(\varpi_1,\varpi_2)=\varpi_1\varpi_2=-\alpha_2-\alpha_1(\alpha_1\alpha_2-\alpha_3)(\alpha_2-\alpha_1^2)^{-1}.
%\eal
%\een
%Then
%\begin{equation}
%X^{[3]}-\omega_1X^{[2]}-X^{[2]}*(\varpi_1+\varpi_2)I+\omega_1X*(\varpi_1+\varpi_2)I+X*(\varpi_1\varpi_2)I-\omega_1\varpi_1\varpi_2=0.
%\end{equation}
But we don't know that this equality holds in general. When $q\rightarrow 1$, the equation \eqref{q-super-cay-ham} is the super Cayley-Hamilton theorem proved in \cite{UM,UM2}. When $n=0$, the equation \eqref{q-super-cay-ham} specializes to the quantum Cayley-Hamilton theorem in \cite{Zh,JLZ}.
\end{remark}

\begin{proposition}\label{alpha-sym}
For $k>0$,
\beq
\alpha_k=\mathbb S_{(1^k)}(\omega_1,\ldots,\omega_m,-\varpi_1,\ldots,-\varpi_n).
\eeq
\end{proposition}
\begin{proof}
It follows from lemma \ref{w-solution} with \eqref{char} and the definition of Schur supersymmetric polynomials that  $\Gamma(t)$ can be written by supersymmetric polynomials:
\beq
\bal
\Gamma(t)&=\frac{\prod_{i=1}^m(t-\omega_i)}{\prod_{j=1}^n(t-\varpi_j)}\\
&=\frac{\sum_{i=0}^{m}(-1)^i e_i(\omega_1,\ldots,\omega_m)t^{m-i}}
{\sum_{j=0}^{n}(-1)^j e_j (\varpi_1,\ldots,\varpi_n)t^{n-j}}\\
&=\sum_{k=0}^{\infty}(-1)^{k} \mathbb S_{(1^k)}(\omega_1,\ldots,\omega_m,-\varpi_1,\ldots,-\varpi_n)t^{m-n-k}.
\eal
\eeq
\end{proof}

The quantum super immanants of generator matrix of quantum coordinate superalgebra  can be expressed as a Schur supersymmetric polynomials at particular arguments, which can be viewed as a quantum super analog of Littlewood correspondence III \cite{Li}[p. 121].
\begin{theorem}\label{t:Litt3}
   Let $\lambda \vdash r$ and $\omega_1,\ldots,\omega_m, \varpi_1,\ldots,\varpi_n$ be the solution in lemma \ref{w-solution}, then
\begin{equation}\label{Littlewood3}
\mathbb{S}_{\lambda}(\omega_1,\ldots,\omega_m,-\varpi_1,\ldots,-\varpi_n)
=str_{1,\ldots,r}\Ec^{\lambda} X_1\ldots X_r= \sum_{I}\frac{ \Imm_{\chi_q^{\lambda}}(X_I)}{\alpha_{\mathrm{q}^2}(I)},
\end{equation}
where the sum is taken over all non-decreasing ordered multisets $I$ of $[m+n]$  satisfying $|I|=r$.
In particular,
\beq
\bal
&\beta_r=\mathbb{S}_{r}(\omega_1,\ldots,\omega_m,-\varpi_1,\ldots,-\varpi_n),\\
&\alpha_r=\mathbb{S}_{(1^r)}(\omega_1,\ldots,\omega_m,-\varpi_1,\ldots,-\varpi_n),\\
&\gamma_r=strX^{[r]}=p_{m,n}^{(r)}(\omega_1,\ldots,\omega_m,-\varpi_1,\ldots,-\varpi_n).
\eal
\eeq
\end{theorem}
\begin{proof}
It follows from proposition \ref{alpha-sym} and $q$-super Goulden-Jackson identities \eqref{GJ-identity} that
\beq
\bal
&\mathbb{S}_{\lambda}(\omega_1,\ldots,\omega_m,-\varpi_1,\ldots,-\varpi_n)\\
&=\sum_{\mu}K_{\lambda,\mu}^{-1}\beta_{\mu_1}\cdots \beta_{\mu_{\lambda_1}}\\
&=\sum_{\mu}K_{\lambda^{T},\mu}^{-1}\alpha_{\mu_1}\cdots \alpha_{\mu_{\lambda_1}}\\
&=\sum_{I}\frac{ \Imm_{\chi^{\lambda}}(X_I)}{\alpha(I)}.
\eal
\eeq

Finally, by the Newton's identities in theorem \ref{Newton-iden}, $\gamma_r=strX^{[r]}$ can be expressed as the same linear combination of $\alpha_{\mu}, \mu \vdash r$ as the power sum supersymmetric polynomial $p_{m,n}^{(r)}$.
\end{proof}

Let $\mathcal{X}_m = (x_1, \ldots, x_m)$, $\mathcal{Y}_n = (y_1, \ldots, y_n)$ be two sets of indeterminates, the set of all supersymmetric polynomials is called the  algebra of supersymmetric polynomials in $\mathcal{X}_m, \mathcal{Y}_n$. We will denote it by $\textbf{Sym}^{(m|n)}$.

Consider the algebra homomorphism $\Phi: A_q(\Mat_{m|n}) \rightarrow \mathbb{C}[\mathcal{X}_m, \mathcal{Y}_n]$ defined as
\ben
\bal
&x_{ij}\mapsto 0, \quad i\neq j, \\
&x_{ii}\mapsto x_i, \quad 1 \leq i \leq m,\\
&x_{jj}\mapsto -y_j, \quad  1 \leq j \leq n.
\eal
\een
\begin{theorem}\label{iso-sym} The restriction of $\Phi$:
  $$\Phi |_{\mathfrak{B}_{m|n}} :\mathfrak{B}_{m|n}\rightarrow \textbf{Sym}^{(m|n)}$$
  is an algebra isomorphism. And
  $\{ \sum_{I}\frac{ \Imm_{\chi_q^{\lambda}}(X_I)}{\alpha_{\mathrm{q}^2}(I)}\mid \lambda \in H(m,n)\}$ is a basis of $\mathfrak{B}_{m|n}$.
\end{theorem}

\begin{proof}
It is easy to see that $\Phi(\alpha_k)=\mathbb{S}_{(1^k)}(\mathcal{X}_m, -\mathcal{Y}_n)$. So the map $\Phi |_{\mathfrak{B}_{m|n}}$ is an epimorphism by the fundamental theorem of supersymmetric polynomials. Moreover, $\Phi|_{\mathfrak{B}_{m|n}}$ is an isomorphism because of the Littlewood correspondence III in theorem \ref{t:Litt3}. So $\{ \sum_{I}\frac{ \Imm_{\chi_q^{\lambda}}(X_I)}{\alpha_{\mathrm{q}^2}(I)}\mid \lambda \in H(m,n)\}$ is a basis of $\mathfrak{B}_{m|n}$.
\end{proof}

\begin{theorem}\label{rela-SchurPower}
Let $\lambda \vdash r$, then
$$\sum_{|I|=r}\frac{ \Imm_{\chi_q^{\lambda}}(X_I)}{\alpha_{\mathrm{q}^2}(I)}=\frac{\Imm_{\chi^{\lambda}} (\Gamma_r)}{r!},$$
where the immanants on the  right hand side are the classical immanants for non-super matrices  and the lower Hessenberg matrix  $\Gamma_r$ is
$$
\Gamma_r=\begin{pmatrix}
\gamma_1 & 1 &  & &  \\
\gamma_2 & \gamma_1 & 2 &  & \\
\gamma_3 & \gamma_2 & \gamma_1 & 3 & \\
\vdots & \vdots & \ddots & \ddots &  \ddots\\
\gamma_r & \gamma_{r-1} & \cdots & \cdots & \gamma_1
\end{pmatrix}.$$
\end{theorem}
\begin{proof}
By the Newton identities in Theorem \ref{Newton-iden} and Cramer's rule, we have that
\[
\alpha_{r}=\frac{\det(\Gamma_r)}{r!}.
\]
Moreover, the Schur supersymmetric polynomials can be written as
\[
r!\ \mathbb{S}_{\lambda}=\Imm_{\chi^{\lambda}}(\mathfrak{P}_r),
\]
where the matrix $\mathfrak{P}_r$ is  obtained by replacing  $\gamma_i$ with the power sum supersymmetric  polynomial $p_{m|n}^{(i)}$ in matrix $\Gamma_r$. According to the isomorphism $\Phi|_{\mathfrak{B}_{m|n}}$ and  the quantum super Goulden-Jackson identities, we have that
\ben
\bal
\sum_{|I|=r}\frac{ \Imm_{\chi_q^{\lambda}}(X_I)}{\alpha_{\mathrm{q}^2}(I)}=
\sum_{\mu}K_{\lambda^{T},\mu}^{-1}\alpha_{\mu_1}\cdots \alpha_{\mu_{\lambda_1}}
=\frac{\Imm_{\chi^{\lambda}}(\Gamma_r)}{r!}.
\eal
\een
\end{proof}

\bibliographystyle{amsalpha}

\end{document}